\documentclass{amsart}

\usepackage{graphicx}        
\usepackage{multicol}        
\usepackage[bottom]{footmisc}

\usepackage{float}
\usepackage{subfig}

\newcommand{\pa}{\partial}
\newcommand{\lt}{\left}
\newcommand{\rt}{\right}
\newcommand{\R}{\mathbb{R}}

\usepackage{mathtools}
\usepackage{graphicx}

\newtheorem{theorem}{Theorem}[section]

\theoremstyle{definition}
\newtheorem{definition}[theorem]{Definition}

\theoremstyle{remark}
\newtheorem{remark}[theorem]{Remark}

\theoremstyle{corollary}
\newtheorem{corollary}[theorem]{Corollary}

\makeindex             

\begin{document}

\title[  ]{A review on attractive-repulsive hydrodynamics for consensus in collective behavior}

\author{Jos\'e A. Carrillo}
\address[Jos\'e A. Carrillo]{Department of Mathematics, Imperial College London, SW7 2AZ, UK}
\curraddr{}
\email{carrillo@imperial.ac.uk}
\thanks{}

\author{Young-Pil Choi}
\address[Young-Pil Choi]{Fakult\"at f\"ur Mathematik, Technische Universit\"at M\"unchen, Boltzmannstra{\ss}e 3, 85748, Garching bei M\"unchen, Germany}
\curraddr{}
\email{ychoi@ma.tum.de}
\thanks{}

\author{Sergio P\'erez}
\address[Sergio P\'erez]{ETSIAE, Technical University of Madrid, Pza. de Cardenal Cisneros, 3, 28040, Madrid, Spain}
\curraddr{}
\email{sergio.perez.perez@alumnos.upm.es}
\thanks{}

%
%

\begin{abstract}
This survey summarizes and illustrates the main qualitative properties of hydrodynamics models for collective behavior. These models include a velocity consensus term together with attractive-repulsive potentials leading to non-trivial flock profiles. The connection between the underlying particle systems to the swarming hydrodynamic equations is performed through kinetic theory modelling arguments. We focus on Lagrangian schemes for the hydrodynamic systems showing the different qualitative behavior of the systems and its capability of keeping properties of the original particle models. We illustrate known results concerning large time profiles and blow-up in finite time of the hydrodynamic systems to validate the numerical scheme. We finally explore unknown situations making use of the numerical scheme showcasing a number of conjectures based on the numerical results.
\end{abstract}

\maketitle


\section{Introduction}\label{sec:1}
Modelling the collective behavior of a large number of interacting individuals is a very challenging problem in animal behavior, pedestrian flow, cell adhesion and chemotaxis problems, and many other biological applications, see for instance \cite{CK,PEK,CDFSTB,BD,BMP,PBSG} and the literature therein. Most of the literature is based on Individual Based Models (IBMs) which are particle descriptions from a kinetic modelling perspective. These particle systems typically include three basic effects: attraction, repulsion and alignment or re-orientation of the individuals, called the first principles of swarming. The way in which these three effects are taken into account has given rise to a large number of different and interesting models for collective behavior. These basic 3-zone models were introduced by theoretical biologists \cite{Aoki,HW} for fisheries control as well as computer scientists \cite{Reyn} in order to mimick animal behavior in animation movies. These models have evolved toward more complete descriptions involving particular species interactions and adapted to particular animals such as birds \cite{HCH}, fish \cite{HH,KTIHC,BTTYB}, ducks \cite{LLE,LLE2}, and insects \cite{BT} for instance.

These basic particle descriptions can be coarsened to macroscopic descriptions when the number of individuals is large leading to nonlocal macroscopic models both at the level of the mass density \cite{MEK,MEBS} or hydrodynamic descriptions \cite{CDMBC,CDP}. This connection to continuum models is better done by passing to the intermediate description provided by kinetic modelling. The kinetic theory approach via mean-field limits of interacting particle systems has offered a mathematical underpinning to derive kinetic equations in a rigorous manner from particle descriptions. The connection towards macroscopic equations is done either via closure assumptions or moment approximations \cite{CDMBC,CDP} and large friction limits \cite{LT}. One of the most famous particle models was introduced by Vicsek and his collaborators \cite{VCBCS} showing a phase transition behavior that has also been studied through kinetic modelling and self-organized hydrodynamics \cite{dm2008,dfl2013,dfl2015}.
We refer to \cite{CFTV,KCBFL} and the references therein for a good account of the different levels of description and the state of the art of these models in the applied mathematical community.

We will focus in this review on two velocity consensus models \cite{CS0,CS,MT,MT2} that lead to asymptotic convergence for the large time toward a fixed velocity under certain conditions, a phenomenon that is called asymptotic flocking. These models have been studied extensively in the last years due to their apparent simplicity in formulating the possibility of consensus in velocity. These models are connected to the Vicsek model in which all particles travel to a fixed speed by large friction limits \cite{BC}. They also present a phase transition in terms of noise as the original Vicsek model \cite{BCCD}. In this survey, we concentrate in the basic properties of the hydrodynamic models incoporating also the effects of attraction and repulsion through an interaction potential. 

In Section 2, we give a brief account of the particle descriptions making particular emphasis to the consensus in velocity models with interaction potentials and their flock solutions. Section 3 is first devoted to explain the link between these particle models and hydrodynamic descriptions via kinetic modelling. We propose a Lagrangian approach to solve the one dimensional hydrodynamic descriptions. We numerically explore different qualitative aspects of the hydrodynamic models such as critical thresholds \cite{CCTT} and their sharpness for consensus models with and without interaction potentials. We also analyse the effect of the singularity of the potential in the long time asymptotics of global solutions. 
%
%
%
%
%
%

\section{Microscopic descriptions: Discrete models}
In this section, we review some of the basic individual based attractive-repulsive models containing an additional velocity alignment force. The social interaction between individuals of the swarm is modelled by an effective interaction potential encapsulating the short-range repulsion and the long-range attraction forces at the particle level as discussed in the introduction. On top, we will also consider cases in which there is a tendency of behaving similarly to other individuals of the group, this mimicking behavior can be modelled in many different ways. One of the simplest manners of incorporating this gregarious behavior is to assume that each individual averages its relative velocity vector with nearby individuals according to some weights that we call the communication function. All the modelling in these simple descriptions are reduced to find biologically reasonable potentials and communication functions for the particular application or adapted to a particular species. Many authors have studied what are the most probable interaction regions for different animals, see \cite{HH,LLE2,KTIHC} and the references therein. We will showcase some of the different behaviors in these models by choosing toy-example for potentials and communication functions. Although not too biologically reasonable, these choices give us generic behaviors for these models.

For the velocity alignment force, we use two different types of forces proposed by Cucker and Smale \cite{CS0,CS} and Motsch and Tadmor \cite{MT}. More precisely, let $(x_i,v_i)$ be the position and velocity of $i$-th individual. Then our main system reads as
\begin{equation}\label{eq_par}
\begin{cases}
{\displaystyle \frac{dx_{i}}{dt}=v_{i},\quad i = 1,\cdots,N,\quad t>0,}\\[2mm]
{\displaystyle \frac{dv_{i}}{dt}=\frac{1}{S_i(x)}\sum_{j=1}^{N}\psi\left(x_{i}-x_{j}\right)\left(v_{j}-v_{i}\right) - \frac1N\sum_{j \neq i} \nabla K(x_i - x_j).}
\end{cases}
\end{equation}
The first term on the right hand side of $\eqref{eq_par}_2$ represents a nonlocal velocity alignment force, where $\psi$ is the communication function. The second term on the right hand side of $\eqref{eq_par}_2$ serves as attractive/repulsive forces through the interaction potential $K(x)$. Typical assumptions on $K(x)$ are that is radially symmetric and smooth outside the origin possibly decaying to zero for large distances, one particular example widely used in the literature is the Morse potential, see \cite{DCBC,CMP,ABCV,CHM1} for more details. Here the scaling function $S_i(t)$ and the communication function $\psi$ are given by
\begin{equation}\label{def_s}
S_i(x) := \left\{ \begin{array}{ll}
N & \mbox{ for the Cucker-Smale model},\\[1mm]
\displaystyle \sum_{k=1}^N \psi(x_i - x_k) & \mbox{ for the Motsch-Tadmor model},
\end{array}\right. 
\end{equation}
and
\[
\psi(x)=\frac{1}{(1+|x|^{2})^{\beta/2}},\quad\beta\geq 0,
\]
respectively. These scalings are related to the mean-field limit for the system of the N interacting particles. Assuming that the effect of each individual on another one via the social force decays as $1/N$ is intuitive, if we want to obtain some non trivial limit as $N\to \infty$, since we should keep the total kinetic and potential energy and velocity of each individual to be of order 1 in that limit. More discussions about the mean-field limit can be found in \cite{BH,Dob,Spo,Szn,Gol,Hau,CCR,CCR2,BCC,CFTV,AIR,BV,FHM,CCH,CCH2}.

In \cite{CS}, Cucker and Smale introduced the Newton-type particle system \eqref{eq_par} for flocking phenomena. The local averaging of relative velocities is weighted by the communication function $\psi$ in such a way that closer individuals have stronger influence than further ones. Note that the velocity alignment force of the Cucker-Smale (in short CS) model is scaled with the total mass. Later, Motsch and Tadmor proposed in \cite{MT} a new model for self-organized dynamics. They pointed out that the CS model is inadequate for far-from-equilibrium scenarios since the communication function is normalized by the total number of agents $N$.  By taking into account the velocity-alignment force normalized with a local average density, the Motsch-Tadmor (in short MT) model takes into account not only the relative distance between agents but also their relative weights compared to the CS model. Note that the MT model does not have the symmetry property due to the normalization.

%
%
%
%
%
%

\subsection{Velocity-alignment models without interaction forces}\label{sec_p1}

We begin our discussion with the case in which the individuals are only interacting through the velocity alignment force as
\begin{equation}
\begin{cases}
{\displaystyle \frac{dx_{i}}{dt}=v_{i},\quad i = 1,\cdots,N,\quad t>0,}\\[2mm]
{\displaystyle \frac{dv_{i}}{dt}=\frac{1}{S_i(x)}\sum_{j=1}^{N}\psi\left(x_{i}-x_{j}\right)\left(v_{j}-v_{i}\right),\quad\psi(x)=\frac{1}{(1+|x|^{2})^{\beta/2}},\quad\beta\geq0,}
\end{cases}\label{eq_cs}
\end{equation}
with the initial data 
\begin{equation}
(x_{i}(0),v_{i}(0))=:(x_{i0},v_{i0}),\quad i = 1,\cdots,N.\label{ini_eq_cs}
\end{equation}
Here the scaling function $S_i(x)$ is given in \eqref{def_s}. We notice that the standard Cauchy-Lipschitz theory yields the existence and uniqueness of global in time smooth solutions to the system \eqref{eq_cs} with $S_i(x) \equiv N$ since the communication function $\psi$ is bounded and globally Lipschitz. For the MT model, we can also show that the communication function $\psi$ is bounded from below for any time $T < \infty$, and this again enables us to apply the Cauchy-Lipschitz theory to the MT model to have the existence and uniqueness of solutions. Let us first remind the main analytical results concerning the flocking behavior for the system \eqref{eq_cs}. Then, we present several numerical results to illustrate the analytical ones. We also compare the time behavior of solutions to the CS and MT systems, i.e., \eqref{eq_cs} with $S_i(x) \equiv N$ and $S_i(x) = \sum_{k=1}^N \psi(x_i - x_k)$.

For the large-time behavior of solutions, we first introduce the definition of universal asymptotic flocking for the system \eqref{eq_cs}.

\begin{definition}
Let $(x_i,v_i)_{i=1}^N$ be a given solution of the particle system \eqref{eq_cs}-\eqref{ini_eq_cs}. Then the $(x_i,v_i)_{i=1}^N$ leads to asymptotic flocking if and only if it satisfies the following two conditions:
\[
\lim_{t \to \infty}\max_{1 \leq i,j \leq N}|v_i(t) - v_j(t)| = 0 \quad \mbox{and} \quad \sup_{0 \leq t < \infty}|x_i(t) - x_j(t)| < \infty.
\]
\end{definition}
We then define diameters in position and velocity phase spaces as follows: 
\begin{equation}
R^{x}(t):= \max_{1\leq i,j\leq N}\left|x_{i}(t)-x_{j}(t)\right|,\quad R^{v}(t):=\max_{1\leq i,j\leq N}\left|v_{i}(t)-v_{j}(t)\right|.\label{def:xv}
\end{equation}

For the system \eqref{eq_cs}, rigorous estimates showing the emergence of flocking depending on the decay rate of the communication function are provided in \cite{CS0,CS}. Later, the flocking estimates are refined in \cite{CFRT, HL, HT, MT, Tan}. Flocking models with vision cones or topological interactions are studied in \cite{AP, CCHS, Hasko} and with noise \cite{Choi, DFT}.

In the theorem below, we sumarize the flocking estimates for the system \eqref{eq_cs}. The proof follows the blueprint of \cite{CFRT, ACHL, HL, MT, Tan}, so we omit it here.

\begin{theorem}\label{thm_fl}Let $(x,v)$ be any global smooth solution to the CS system \eqref{eq_cs}-\eqref{ini_eq_cs}.
\begin{itemize}
\item If $0\leq\beta\leq1$, then we have unconditional asymptotic flocking, that is
\begin{equation}
R^{v}(0)e^{-t}\leq R^{v}(t)\leq R^{v}(0)e^{-\psi(\tilde{R})t}\quad t\geq0,
\label{new}
\end{equation}
where $\tilde{R}$ is implicitly given by 
\begin{equation*}
R^{v}(0)=\int_{R^{x}(0)}^{\tilde{R}}\varphi(s)\,ds.
\end{equation*}

\item If $\beta>1$ and the initial diameters $R^{x}(0)$ and $R^{v}(0)$
satisfy 
\begin{equation}
R^{v}(0)<\int_{R^{x}(0)}^{\infty}\psi(s)\,ds,\label{a:gene-1}
\end{equation}
then estimate \eqref{new} also holds. 
\end{itemize}
\end{theorem}

In Figs. \ref{fig:sim-beta08} and \ref{fig:beta>1}, we observe typical particle simulations of the CS model in 2D. As stated in Theorem \ref{thm_fl}, the unconditional asymptotic flocking occurs for any initial data in the case simulated in Fig. \ref{fig:sim-beta08} while Fig. \ref{fig:beta>1} shows a comparison between flocking and non flocking cases. 
\begin{figure}[ht!]
\subfloat[$Time=0$]{\protect\protect\includegraphics[scale=0.32]{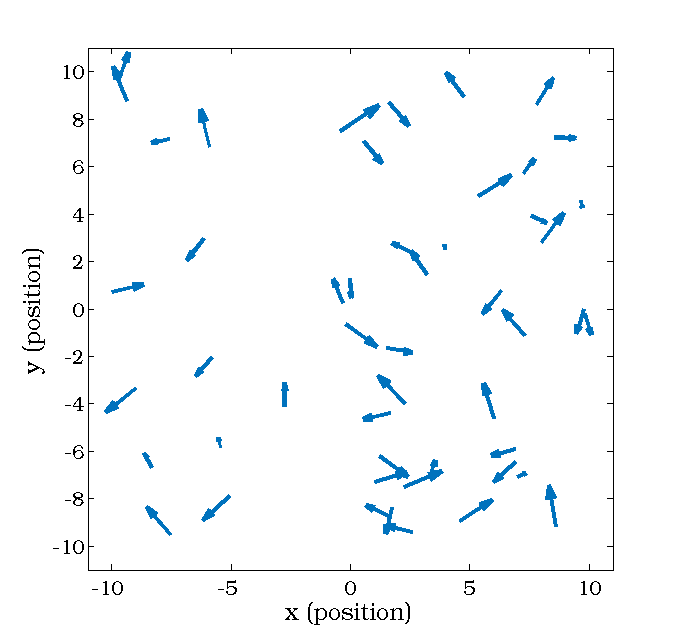}
}\subfloat[$Time=250$]{\protect\protect\includegraphics[scale=0.32]{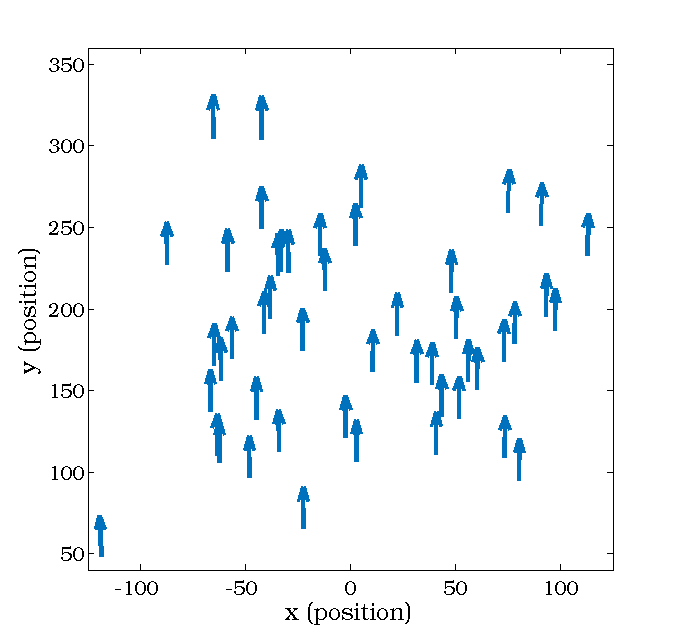}
}
\protect\protect\caption{\label{fig:sim-beta08}Large-time behavior of solutions to the CS model with $\beta=0.8$.
}
\end{figure}
In Fig. \ref{fig:sim-beta08}, the initial positions and velocities $\{(x_{i0},v_{i0})\}_{i=1}^N$ with $N=50$ are generated randomly from the uniform distribution $[-10, 10]^2 \times \big\{ [-5,5]\times [-4.3,5.7] \big\}$ with the aim of having the mean velocity equal to $\frac{1}{N} \sum_{i=1}^N v_i(0) \approx (0,0.7)$. 

In Fig. \ref{fig:beta>1}, the flocking behavior happens depending on the initial data. The initial positions and velocities $\{(x_{i0},v_{i0})\}_{i=1}^N$ with $N=50$ are generated randomly from the uniform distribution $[-10, 10]^2 \times \big\{ [-5,5]\times [-4.3,5.7] \big\}$ with the aim of having the mean velocity equal to $\frac{1}{N} \sum_{i=1}^N v_i(0) \approx (0,0.7)$. With this initial configuration, it results that $R^x(0)=26.23$ and $R^v(0)=12.25$. Then, the initial data for (A) satisfies \eqref{a:gene-1} since $R^{v}(0)<\int_{R^{x}(0)}^{\infty}\psi(s)\,ds=16.43$. On the other hand, the initial data for (B) do not satisfy the condition \eqref{a:gene-1} because $R^{v}(0)>\int_{R^{x}(0)}^{\infty}\psi(s)\,ds=2.60$.
\begin{figure}[ht]
\subfloat[$Time=250000$]{\protect\protect\includegraphics[scale=0.32]{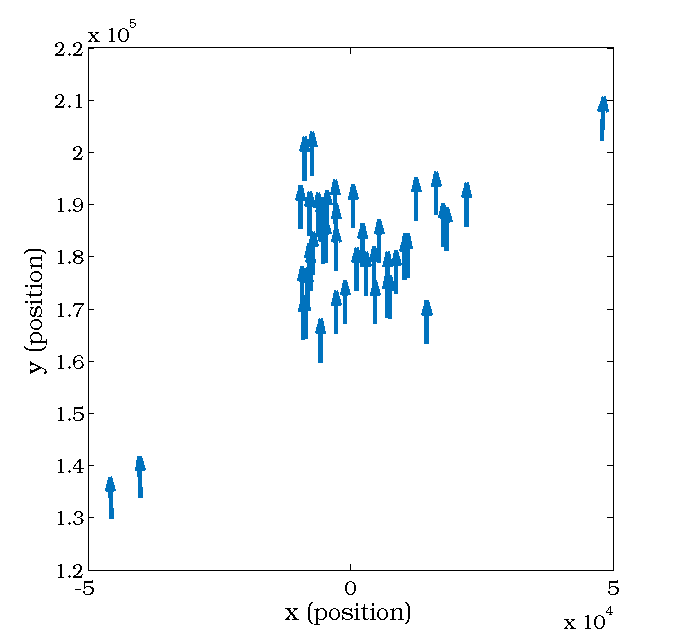}
}\subfloat[$Time=250000$]{\protect\protect\includegraphics[scale=0.32]{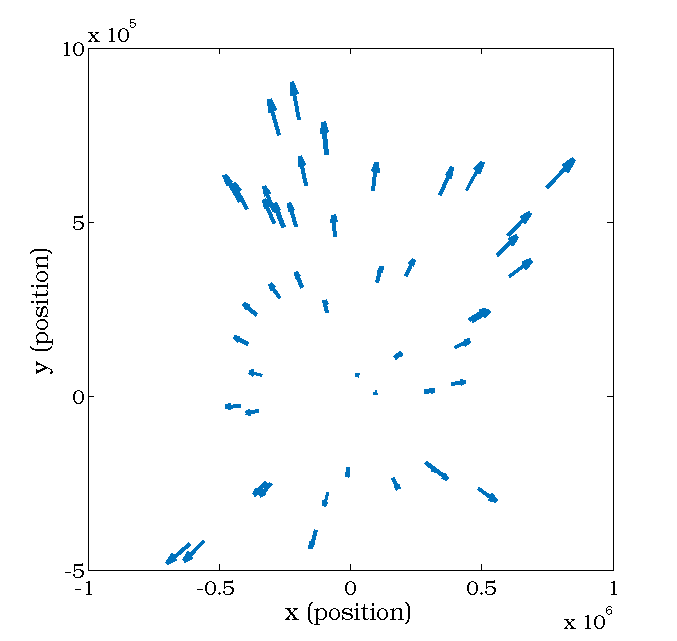}
}
\protect\protect\caption{\label{fig:beta>1}Large-time behavior of solutions to the CS model with $\beta=1.05$ (A) and $\beta=1.2$ (B).}
\end{figure}
\begin{remark} For the CS model, i.e., $S_i(x) \equiv N$ in \eqref{eq_cs}, 
if we set an averaged quantity $v_c(t):= \frac1N \sum_{i=1}^N v_i(t)$, then $v_c(t)$ satisfies $v_c^\prime(t) = 0$, i.e., $v_c(t) = v_c(0)$ due to the symmetry of the communication function $\psi$. Thus, if the global flocking occurs, then we have that for all $i \in \{1, \cdots, N\}$
\[
v_i(t) \to v_c(0) = \frac1N\sum_{i=1}^Nv_i(0) \quad \mbox{as} \quad t \to \infty.
\]
On the other hand, in the MT model, i.e., $S_i(x) = \sum_{k=1}^N \psi(x_i - x_k)$, the momentum is not conserved. Thus identifying the asymptotic flocking state in terms of the initial data  is a very intriguing question. Partial answers to asymptotic flocking have been provided in \cite{MT}.
\end{remark}

In Fig. \ref{fig:Initial-conditions-mt}, we show the different behavior of the CS and MT velocity averaging. We choose the initial positions $\{x_{i0}\}_{i=1}^{55}$ divided into two groups, $G_1 := \{x_{i0}\}_{i=1}^{50}$ and $G_2 := \{x_{i0}\}_{i=1}^5$, they are generated randomly from the uniform distribution $[-10,10]^2$ and $[60,63] \times [-1.5,1.5]$, respectively. The initial velocities $\{v_{i0}\}_{i=1}^{55}$ are generated randomly from the uniform distribution $[-5,5]\times [-4.3,5.7] $ with the aim of having the mean velocity equal to $\frac{1}{N} \sum_{i=1}^N v_i(0) \approx (0,0.7)$. We observe the much faster decay of the velocity radius of the support $R^v(t)$ defined in \eqref{def:xv} in the MT model compared to the CS model, and thus the asymptotic flocking is achieved faster in the MT model than in the CS model.
\begin{figure}[ht!]
\subfloat[]{\protect\protect\includegraphics[scale=0.32]{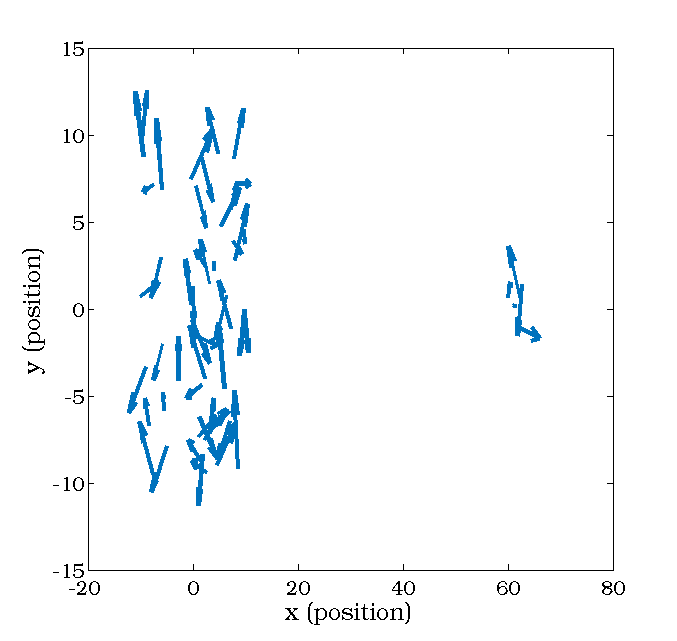}
}\subfloat[]{\protect\protect\includegraphics[scale=0.32]{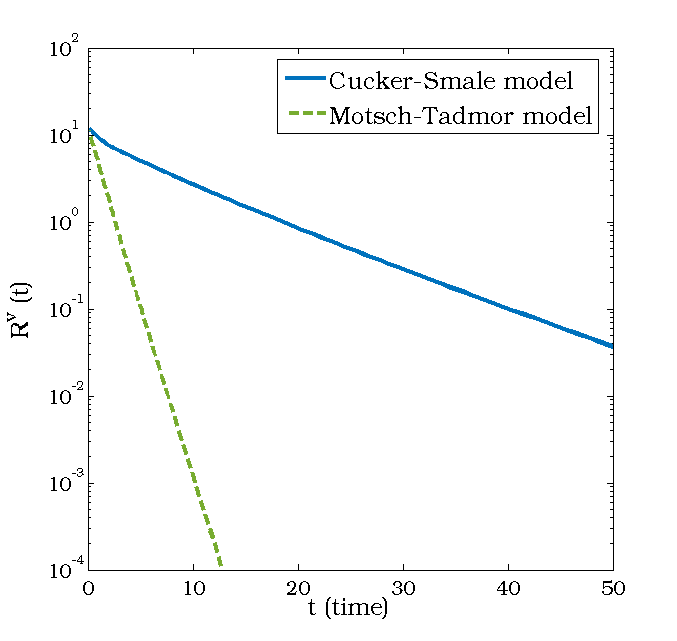}
}
\protect\protect\caption{\label{fig:Initial-conditions-mt} (a): The initial positions are divided into two groups. (b): Comparison of the log scale decay rate of $R^v(t)$ for both systems.
}
\end{figure}

Finally, we show in Fig. \ref{fig:Evolution-mt-5s} a comparison between the time evolution of the system with $S_i(x) \equiv N$ and $S_i(x) = \sum_{k=1}^N \psi(x_i - x_k)$. Subplots (a) and (b) show a snapshot of the solutions at $t=5$, while subplots (c) and (d) show a snapshot of the solutions at $t=50$, for the CS and the MT models respectively. From Fig. \ref{fig:Evolution-mt-5s} (a), we find that the CS flocking particles in the small group $G_2$ are not interacting with others in the large group $G_1$ in the beginning. It seems that the particles tend to move with their own initial velocities. On the other hand, the particles in the MT model are trying to be aligned with their neighbors from the beginning. Even though both models exhibit flocking behavior and the analytical results require the same conditions for flocking, numerical simulations demonstrate again that the decay rate of convergence of the MT model to the flocking state is faster than the one of the CS model. We are not aware of results comparing the rate of decay to flocking for both models.
\begin{figure}[ht!]
\subfloat[]{\protect\protect\includegraphics[scale=0.32]{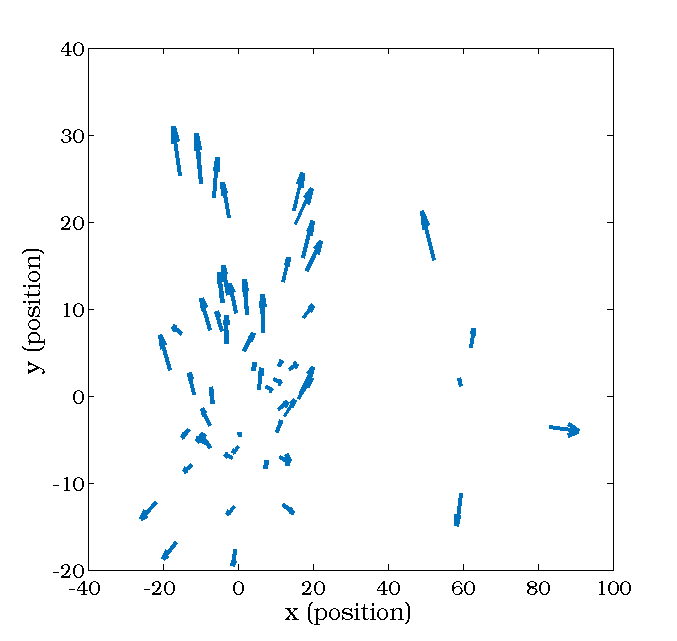}
}\subfloat[]{\protect\protect\includegraphics[scale=0.32]{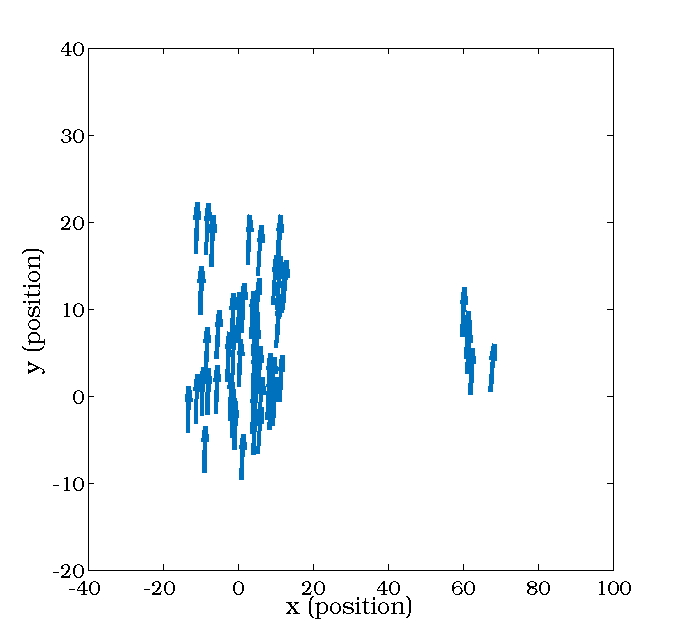}
}

\subfloat[]{\protect\protect\includegraphics[scale=0.32]{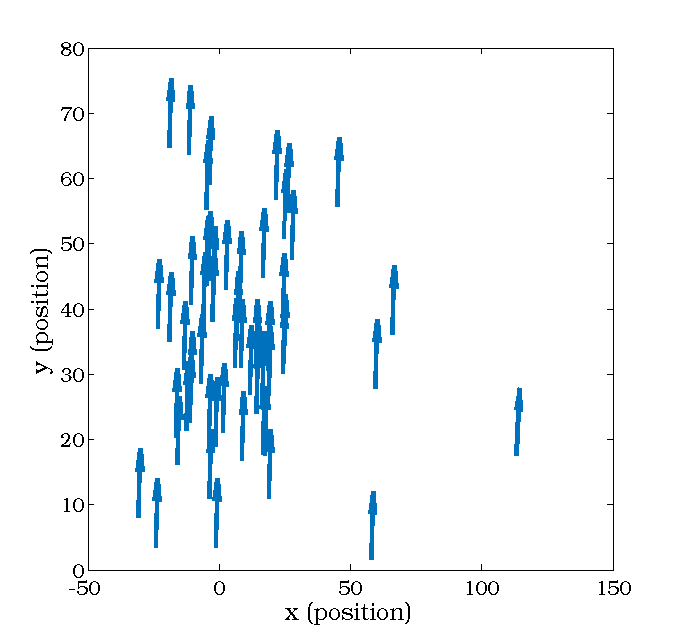}
}\subfloat[]{\protect\protect\includegraphics[scale=0.32]{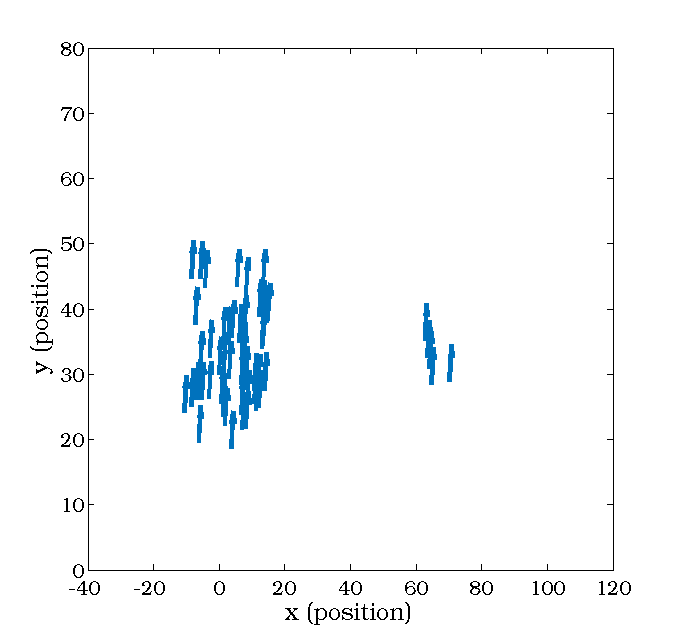}
}
\protect\protect\caption{\label{fig:Evolution-mt-5s}Comparison between time evolution of the system with $S_i(x) \equiv N$ and $S_i(x) = \sum_{k=1}^N \psi(x_i - x_k)$. Subplots (a) and (b) show a snapshot of the solutions at $t=5$ and subplots (c) and (d) show a snapshot of the solutions at $t=50$, for the CS and the MT models respectively.
}
\end{figure}
\vspace{2cm}

%
%
%
%
%
%
\subsection{Attractive-Repulsive models}\label{sec_p2} Now, we turn to the case in which we incorporate attractive-repulsive forces to the system \eqref{eq_par} with the CS alignment force:
\begin{equation}
\begin{cases}
\dot{x}_{i}=v_{i},\quad i = 1,\cdots,N, \quad t > 0,\\[2mm]
\displaystyle \dot{v}_{i}=\frac{1}{N}\,\sum_{j=1}^{N}\psi\left(x_{i}-x_{j}\right)\left(v_{j}-v_{i}\right)-\frac{1}{N}\,\sum_{j=1}^{N}\nabla K\left(x_{i}-x_{j}\right).
\end{cases}\label{cs-simul}
\end{equation}
For the interaction potential $K$, we will choose in most of our simulations a repulsive Newtonian potential confined by a quadratic one of the form
\begin{equation}\label{int_pot}
K(x) = \alpha\frac{\left|x\right|^{2}}{2} + k \phi(x) \quad \mbox{where} \quad -\Delta_x \phi(x) = \delta_0(x).
\end{equation}
By choosing $\alpha >0$, we confine our particles in a bounded region and they are repelled by Newtonian interaction choosing $k<0$. In general, $K$ is typically chosen repulsive at the origin and attractive at infinity in such a way that there is a typical length of stable interaction between two particles. Other popular choices as mentioned above are Morse-like and power-law like potentials as in \cite{DCBC,CMP,CHM1,CH} and the references therein.

The existence of particular solutions, called flock solutions, for the system \eqref{cs-simul} has recently received lots of attention due to the ubiquitous appearance of this kind of solutions in several swarming models.
\begin{definition}
A flock solution of the particle model \eqref{cs-simul} is a spatial configuration $\hat{x}$ with zero net interaction force on every particle, that is:
$$
\sum_{j\neq i}\nabla K(\hat x_i-\hat x_j)=0, \quad i=1,\cdots,N\,,
$$
that translates at a uniform velocity $m_0\in\R^d$, hence $(x_i(t),v_i(t))=(\hat{x_i}-tm_0,m_0)$. 
\label{def-flock} 
\end{definition}
The richness of the qualitative properties of the spatial configurations for the flock solution, also called flock profile, depending on the potential $K$ is quite impressive, see \cite{KSUB}. Other stable patterns were observed for related systems, for instance single or double rotating mills~\cite{LRC,DCBC,CDP,CKMT,CKR}. However, these milling patterns are typically eliminated due to the presence of the CS alignment term. The stability of flock patterns for the particle system \eqref{cs-simul} has recently been established in \cite{ABCV,CHM2}. 

As the total number of individuals gets large, the system of differential equations is difficult to analyse and usually a continuum description based on mean-field limits is adopted, either at the kinetic level for the particle distribution function ~\cite{CDP,CFTV} or at the hydrodynamic level for the macroscopic density and velocities \cite{CDMBC,CDP} as we will discuss in the next section. At the continuum level, the flock profiles are characterized by searching for continuous probability densities or probability measures $\rho$ of particle locations such that the total force acting on each individual balances out. This is equivalent to finding probability densities or measures $\rho$ such that
\begin{equation}\label{eq:basic}
    \nabla K\ast \rho = 0 \quad\mbox{ on supp}(\rho) \,.
\end{equation}
Being the problem posed on the support of the unknown density $\rho$ implies that the equation \eqref{eq:basic} is highly nonlinear. In fact, characterizing the interaction potentials $K$ such that these profiles are continuous or regular in their support is a very challenging question.
Explicit formulas for solutions to \eqref{eq:basic} for particular potentials such as Morse-like  and power-law like potentials are possible due to the particular properties of associated differential operators \cite{LBT,BT,FHK,CMP,CHM1,CH}. In particular, it is known from classical potential theory that the solution to \eqref{eq:basic} in the case of the confined repulsive Newtonian potential \eqref{int_pot} is given by a characteristic of a ball whose radius is determined to have the right total mass of the system. 

In Fig. \ref{fig:cs_simul}, we show a typical simulation for the system \eqref{cs-simul} with interaction potential given in \eqref{int_pot}. The initial positions and velocities $\{(x_{i0},v_{i0})\}_{i=1}^N$ with $N=50$ are generated randomly from the uniform distribution $[-10, 10]^2 \times \big\{ [-5,5]\times [-4.3,5.7] \big\}$ with the aim of having the mean velocity equal to $\frac{1}{N} \sum_{i=1}^N v_i(0) \approx (0,0.7)$. This simulation shows the generic flock formation after some time in which particles distribute more or less uniformly in a certain ball. Of course, we know that as $N\to \infty$ this distribution of particles will be getting closer and closer to the characteristic of the ball, see \cite{BLL,KSUB,BKSUV} for instance.
\vspace{-0.5cm}
\begin{figure}[ht!]
\subfloat[$Time=5\,$]{\protect\protect\includegraphics[scale=0.32]{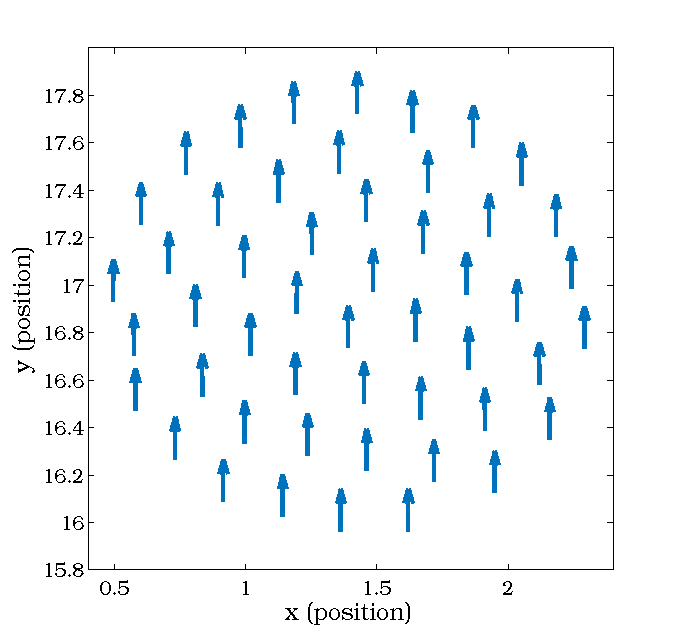}

}\subfloat[]{\protect\protect\includegraphics[scale=0.32]{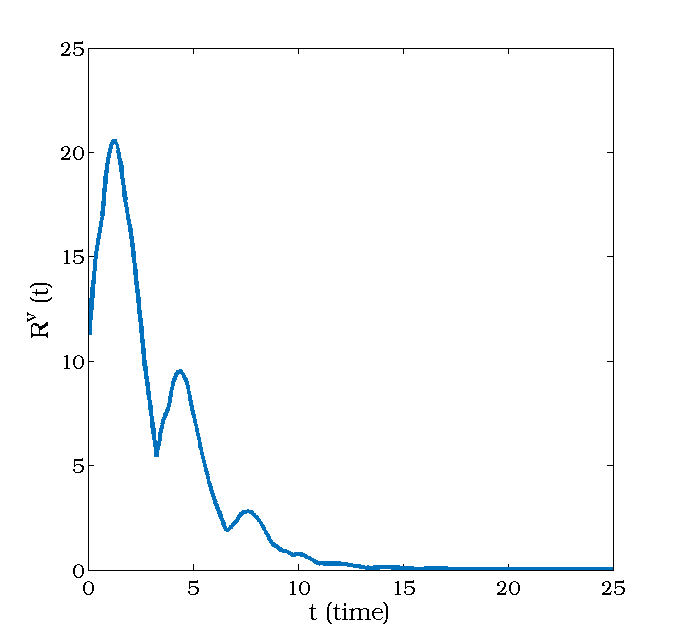}
}
\centering{}\protect\protect\caption{Large time behavior of solutions to the system \eqref{cs-simul} with $k=-1$, $\alpha=1$,  and $\beta = 0.5$. (a): The particle are uniformly distributed in the unit circle. (b): $R^v(t)$ eventually converges to zero as time goes on. However, it is not monotonically decreasing.
}
\label{fig:cs_simul}
\end{figure}

Let us finally mention that the continuum spatial profile of the flock solutions can also be found by steepest descent methods from first-order models of swarming. These models can be found formally from the previous second-order particle models by substituting the CS alignment term by simple linear friction force and assuming that inertia is negligible with respect to other terms, see \cite{MEK,MEBS}. In this limit, they lead to
\begin{equation}\label{1storder}
    \frac{d}{dt}x_i=-\frac{1}{N} \sum_{j\neq i}\nabla K(|x_i-x_j|), \quad i=1,\cdots,N\,.
\end{equation}
By taking the mean-field limit in \eqref{1storder}, as $N\to \infty$, one derives the so-called aggregation equation
\begin{equation}\label{aggreg}
\rho_t = \nabla\cdot(\rho\nabla K*\rho)\,,
\end{equation}
for the evolution of the mass density of particles. The aggregation equation \eqref{aggreg} with repulsive-attractive potentials has attracted lots of attention in the last years in the mathematical analysis community due to its rich regularity structure for steady states and solutions depending on the singularity of the potential at the origin, see \cite{CMV,CMV2,Lau,TB,TBL,BL,BCL,BLR,BLL,BV,BC,BCLR,BCLR2,CDFLS,CDFLS2,CCH} and the references therein.

%
%
%
%

\section{Macroscopic descriptions: Flocking behavior and finite-time blow-up phenomena}
In this section, we study some of the features and properties of continuum models for collective behavior being capable of describing flocking behavior. By using BBGKY hierarchies or mean-field limits \cite{CFTV}, we can derive a Vlasov-type equation from the system \eqref{cs-simul}. More precisely, when the number of individuals goes to infinity, i.e., $N \to \infty$, the mesoscopic observables for the system \eqref{cs-simul} can be calculated from the velocity moments of the density function $f = f(x,v,t)$ which is a solution to the following Vlasov-type equation:
\begin{equation}\label{ki_cs}
\begin{cases}
\partial_{t}f+v\cdot\nabla_{x}f+\nabla_{v}\cdot\left(F(f)f - (\nabla_x K \star \rho) f\right)=0,\quad x,v\in\R^{d}, t>0,\\[2mm]
{\displaystyle \mathbf{\mathrm{\mathit{\mathit{F}(f)= \int_{\mathbb{R}^{d}\times\mathbb{R}^{d}}\psi(x-y)(w-v)f(y,w,t)\,dydw}}}},\\[2mm]
{\displaystyle \rho(x,t) = \int_{\R^d} f(x,v,t)\,dv }.
\end{cases}
\end{equation}
The derivation of the kinetic equation \eqref{ki_cs} is well studied only for regular potentials \cite{BH, Dob, Neun}. If $K$ vanishes, rigorous mean-field limit, existence of weak solutions, and large-time behavior of measure-valued solutions are studied in \cite{CFRT, HL,CCH}. We also refer to \cite{BCHK, BCHK2, BCHK3, BCHK4, CCK, Choi2} for a dynamics of flocking particles interacting with homogeneous/inhomogenous fluids. For the equation \eqref{ki_cs} under certain conditions for $K$, quite general frameworks are proposed in \cite{BCC,CCR2,CCH2}. For not too singular interaction potentials $K$, the rigorous derivation of \eqref{ki_cs} is studied in \cite{HJ}.

The kinetic description has to be taken as usual as an intermediate mesoscopic description of the system leading to macroscopic models for collective behavior via asymptotic limits or closure assumptions. By taking moments on the kinetic equation \eqref{ki_cs} together with a zero temperature closure or monokinetic assumption for the local hydrodynamics solution, one can obtain hydrodynamic descriptions of the system \eqref{cs-simul}. This procedure has been perforemed in different ways by different authors, see for instance \cite{CDP,CDMBC,HT}. Associated to the kinetic distribution function $f(x,v,t)$, one can define the mean velocity as
$$
\rho(x,t) u(x,t) = \int_{\R^d} v f(x,v,t)\,dv\,.
$$
By taking the first two moments with respect to $v$ on the kinetic equation \eqref{ki_cs}, one formaly obtains
\begin{equation}
\begin{cases}
\partial_{t}\rho+\nabla_x\cdot\left(\rho u\right)=0,\quad x\in\R^{d},\quad t>0,\\[2mm]
{\displaystyle \pa_{t}(\rho u)+\nabla_x\cdot(\rho u\otimes u)+ \nabla_x\cdot\left( 
\int_{\R^d} (v-u)\otimes (v-u) f(x,v,t)\,dv \right)=}\\[2mm]
{\displaystyle \qquad\qquad\qquad\qquad\,\,\,\rho\int_{\R^{d}}\psi(x-y)(u(y)-u(x))\rho(y)\,dy - \rho \lt(\nabla_x K \star \rho\rt).}\\
\end{cases}\label{h2_CS}
\end{equation}
Assuming that the distribution function is not far from monokinetic, that is
$$
 f(x,v,t) \simeq \rho (x,t) \, \delta (v - u(x, t)),
$$
the hydrodynamic system \eqref{h2_CS} is reduced to the following pressureless Euler-type equations given by
\begin{equation}
\begin{cases}
\partial_{t}\rho+\nabla_x\cdot\left(\rho u\right)=0,\quad x\in\R^{d},\quad t>0,\\[2mm]
{\displaystyle \pa_{t}(\rho u)\!+\!\nabla_x\!\cdot\!(\rho u\otimes u)\!=\!\rho\!\!\left(\!\int_{\R^{d}}\!\!\psi(x-y)(u(y)-u(x))\rho(y)\,dy \!- \!\nabla_x K \star \rho\!\!\rt),}
\end{cases}\label{h_CS}
\end{equation}
where $\rho = \rho(x,t)$ and $u = u(x,t)$ represent the particle density and their corresponding mean velocity, respectively. For now on, $\Omega(t)$ denotes the interior of the support of the density $\rho$, i.e., $\Omega(t) := \{x \in \R : \rho(x,t) >0\}$. We assume that $\Omega(0)=:\Omega_0$ is a bounded open set. The hydrodynamic system \eqref{h_CS} has to be complemented with initial conditions
\begin{equation}\label{ini_h_CS}
\left.\lt(\rho(\cdot, t),u(\cdot, t)\rt)\right|_{t=0}=\left(\rho_{0},u_{0}\right)\,.
\end{equation}

As usual with hydrodynamic equations, we observe that they are mathematically challenging due to the nonlinearity introduced by the material derivative of the velocity field that can lead to blow-up of the velocity profile. On the other hand, flock profiles are also solutions of the hydrodynamic equations \eqref{h_CS}. Actually, if the density $\rho$ satisfies \eqref{eq:basic} and the velocity is constantly given by $u(x,t)=u_0\in \R^d$, then they form a particular solution of the hydrodynamic equations \eqref{h_CS}. This is the flocking solution at the hydrodynamical level of description for collective behavior.

In the next subsections, we will numerically explore the derived hydrodynamic equations \eqref{h_CS} in one dimension providing numerical evidence showing the flocking behavior and the finite-time blow-up of solutions giving some insight in the stability of flock solutions and the conditions for blow-up of the solutions. Analytical results concerning these hydrodynamic equations are very few in the literature, we will compare to existing analytical results in the relevant sections below.

%
%
%
%

\subsection{Numerical scheme}\label{subsec:numerics} 
For the numerical simulation of the hydrodynamic system \eqref{h_CS}, we use a Lagrangian numerical scheme. With this purpose, we consider the characteristic flow $\eta(x,t)$ associated to the fluid velocity $u$ defined by
\begin{equation}\label{eq:char}
\frac{d \eta(x,t)}{dt} = u(\eta(x,t),t) =: v(x,t) \quad \mbox{with} \quad \eta(x,0) = x.
\end{equation}
Set $h(x,t) := \rho(\eta(x,t),t)$, then using the characteristic flow \eqref{eq:char}, we can rewrite the system \eqref{h_CS} in one dimension as
\begin{align}
\begin{aligned}
\begin{cases}
\displaystyle h(x,t) = \displaystyle \rho_0(x)\lt(\frac{\pa \eta}{\pa x}(x,t)\rt)^{-1},\\[2mm]
\displaystyle \frac{dv}{dt}(x,t) = \displaystyle\int_{\R}\psi(\eta(x,t)-\eta(y,t))\left(v(y,t)-v(x,t)\right)\rho_0(y)\,dy\\[2mm]
\qquad\qquad \displaystyle \,\,\,-\int_{\R} \frac{\partial K}{\partial x}(\eta(x,t)-\eta(y,t))\rho_0(y)\,dy,
\end{cases}\label{eq:eqs-lagrange_notation}
\end{aligned}
\end{align}
with the initial data
\[
\left.\lt(h(\cdot, t),v(\cdot, t)\rt)\right|_{t=0}=\left(\rho_{0},u_{0}\right).
\]
Note that the continuity equation for the density $\eqref{eq:eqs-lagrange_notation}_1$ is decoupled from the rest. Thus it can easily be solved from the information obtained from the characteristic and momentum equations.

We do a spatial discretization of the equation $\eqref{eq:eqs-lagrange_notation}_2$ choosing a uniform mesh of length $\Delta x$ of the initial positions of the particles with a number of points given by $n$. For each node $i$, the position, density and velocity through the characteristics of the ith-particle will be denoted as $\eta_i(t)$, $h_i(t)$ and $v_i(t)$. At each node the term $\frac{\pa \eta_i(t)}{\pa x}$ is computed from $\eta_j(t)$, that is with the information about how the position of the nodes change through the characteristics in time, by standard finite differences of fourth order, and afterwards that value is inverted. The end values use one-sided finite differences to avoid unknown values of the Lagrangian density. The spatial derivative of the interaction potential appearing in $\eqref{eq:eqs-lagrange_notation}_2$ is obtained analytically and not approximated numerically. For the integral terms in $\eqref{eq:eqs-lagrange_notation}_2$, we approximate them by direct numerical quadrature formulas as
$$\begin{aligned}
\int_{\Omega_0}\psi(\eta(x,t)-\eta(y,t)) v(y,t) \rho_0(y)\,dy &\sim\, \Delta x\sum_{j=1,j\neq i}^{n} \psi(\eta_i(t)-\eta_j(t))\,v_j(t)\,\rho_j(0), \cr
\int_{\Omega_0}\psi(\eta(x,t)-\eta(y,t))\rho_0(y)\,dy &\sim\,\Delta x\sum_{j=1,j\neq i}^{n} \psi(\eta_i(t)-\eta_j(t))\,\rho_j(0),\cr
\int_{\Omega_0}  \frac{\partial K}{\partial x}(\eta(x,t)-\eta(y,t))\rho_0(y)\,dy &\sim \,\Delta x\sum_{j=1,j\neq i}^{n} \frac{\partial K}{\partial x}(\eta_i(t)-\eta_j(t))\,\rho_j(0).
\end{aligned}$$
Taking everything into consideration, the system 
of equations \eqref{eq:eqs-lagrange_notation}, for each node $i$, takes the form
\begin{equation}
\begin{cases}
\displaystyle \frac{d\eta_i(t)}{dt}=v_i(t),\\[2mm]
\displaystyle h_i(t)=\rho_i(0)\lt(\frac{\pa \eta_i(t)}{\pa x}\rt)^{-1},\\[2mm]
\displaystyle \frac{dv_i(t)}{dt}\!\!= \Delta x\!\!\!\!\!\!\!\sum_{j=1,j\neq i}^{n} \!\!\!\!\!\!\!\rho_j(0)\!\left[\psi(\eta_i(t)\!-\!\eta_j(t))\!\!\left(v_j(t)\!-\!v_i(t)\right)\!-\!\!\frac{\partial K}{\partial x}\!(\!\eta_i(t)\!-\!\eta_j(t)\!)\right]
\end{cases}\label{eq:spatial-discret}
\end{equation}
A temporal discretization is carried out in the system composed by the equations $\eqref{eq:spatial-discret}_1$ and $\eqref{eq:spatial-discret}_3$, in order to obtain the evolution of position and velocity through the caracteristics. The employed numerical scheme is a classical fourth-order Runge-Kutta explicit scheme using the built-in $ode45$ Matlab command. Subsequently, the density is obtained from $\eqref{eq:spatial-discret}_2$. The numerical experiments have been performed with the version R2014a of Matlab, using a computer ACER Aspire V3-572G. Similar Lagrangian approaches have been taken by other authors in related problems \cite{CKMT,KT}.

In all the test cases below, we choose the same initial conditions for the particle positions, the initial position of the nodes has been set uniformly distributed inside the interval $[-0.75,\,0.75]$. Thus, the initial position of each node $i$ is given by
\[
\eta_i(0)=-0.75+\frac{1.5}{n-1}\left(i-1\right)\quad \mbox{for} \quad i=1,\cdots,n.
\]
In most of our simulations $n=200$ if it is not specified otherwise. The initial conditions of density and velocity will be specified for each case that will be treated subsequently.

%
%
%
%

\subsection{Euler-alignment system}\label{subsec:euler-alig}  The aim of this subsection is to analyze numerically some of the open problems related to the theoretical results studied in \cite{CCTT} on critical threshold phenomena for the system \eqref{h_CS} without the interaction forces. We will take the communication function of CS model given by 
\[
\psi(x)=\frac{1}{(1+|x|^{2})^{\beta/2}}, \quad \beta > 0.
\]
In this case, the global regularity or the finite time blow up of the solution can be determined by the initial configurations with sharp conditions. We use this case as validation to our scheme being capable of showing either the global consensus in velocity of the blow-up in the velocity field and density as shown by the theory \cite{CCTT}. We also numerically compare the large-time behaviors of solutions to the CS and the MT models at the hydrodynamic level. The simulations in this subsection are done with initial density
\[
\rho_{i}(0)=\frac{1}{\gamma}\cos\left(\pi\,\frac{x_{i}(0)}{1.5}\right)\quad \mbox{for each node } i=1,\cdots, n,
\]
where the constant $\gamma$ is fixed by the mass normalization, i.e.,
$\sum_{i=1}^n \rho_{i}(0) = 1$. Concerning the initial velocity, we choose
\[
u_{i}(0)=-c\text{\,}\sin\left(\pi\,\frac{x_{i}(0)}{1.5}\right)\quad \mbox{for each node } i=1,\cdots, n,
\]
where the constant $c >0$ will be varied to study different initial conditions in the simulations except for the comparison between the MT and CS models.

\subsubsection{Hydrodynamic Cucker-Smale model}

In this case, a critical threshold leading to sharp dichotomy between global-in-time existence or finite-time blowup of solutions is provided in \cite{CCTT}. The region where the solutions blow up in a finite time is called ``{\it Supercritical region}'', otherwise the region in which the solutions globally exist in time it called ``{\it Subcritical region}''.   

\begin{theorem}\label{thm_cri} Let $(\rho,u)$ be classical solutions to the system \eqref{h_CS}-\eqref{ini_h_CS} with $K=0$. 
\begin{itemize} 
\item (Subcritical region) If $\partial_{x}u_{0}(x)\geq-\psi\star \rho_{0}(x)$
for all $x\in\mathbb{R}$, then the system has a global classical
solution. 
\item (Supercritical region) If there exists an $x$ such that
$\partial_{x}u_{0}(x)<-\psi\star \rho_{0}(x)$, then there is a finite
time blow up of the solution. Moreover, this blow-up happens as an infinite negative slope in the velocity and divergence value of the density at the same location.
\end{itemize}
\end{theorem}

For the numerical simulations, three different cases corresponding
with the values of the constant $c =0.2, 0.4, 0.5$ are treated. The first two cases lie in subcritical region, and the third one lies in the supercritical region, see Fig. \ref{fig:IC-eu-alig}.
\begin{figure}
\center \includegraphics[scale=0.4]{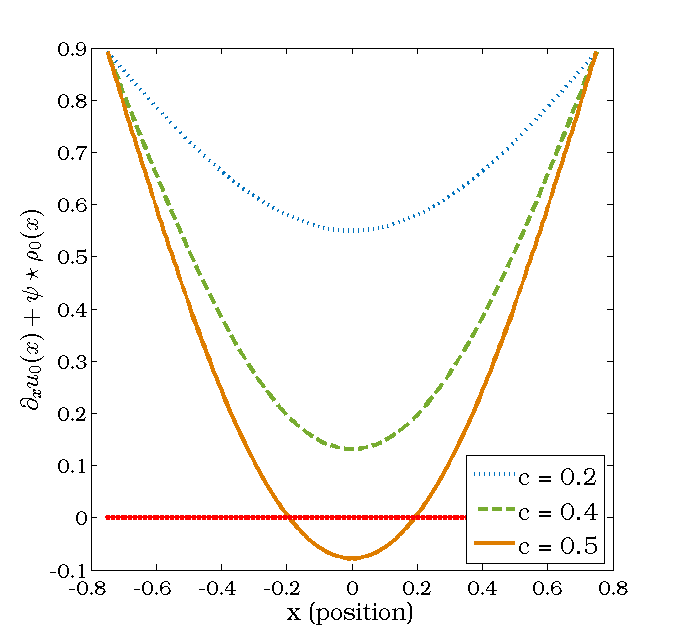}\protect\protect\caption{\label{fig:IC-eu-alig}Values of $\partial_{x}u_{0}(x)+\psi\star \rho_{0}(x)$ for different values of 
$c$.}
\end{figure}

Since the initial configurations of the first two cases lie in the subcritical region, we have global regularity of solutions. Both initial configurations are symmetric, thus the initial mean velocity and center of mass are zero, which are kept through their evolution. The numerical simulations in Fig. \ref{fig:eu-alig-c02} demonstrate that, and they are consistent with Theorem \ref{thm_cri} leading to global in time solutions. Even though both cases, $c = 0.2, 0.4$ are inside the subcritical region, the steady density for the case $c = 0.4$ shows more concentrated behavior at the middle of the domain(Fig. \ref{fig:eu-alig-c02}. (c)) than the case $c = 0.2$(Fig. \ref{fig:eu-alig-c02}. (a)). For both cases, the velocity converges to zero as time goes on which gives the global consensus or flocking behavior in this case. We see that the profile of density depends in a complicated way on the initial density configuration as it happens in the particle system.

\begin{figure}
\subfloat[]{\protect\protect\includegraphics[scale=0.32]{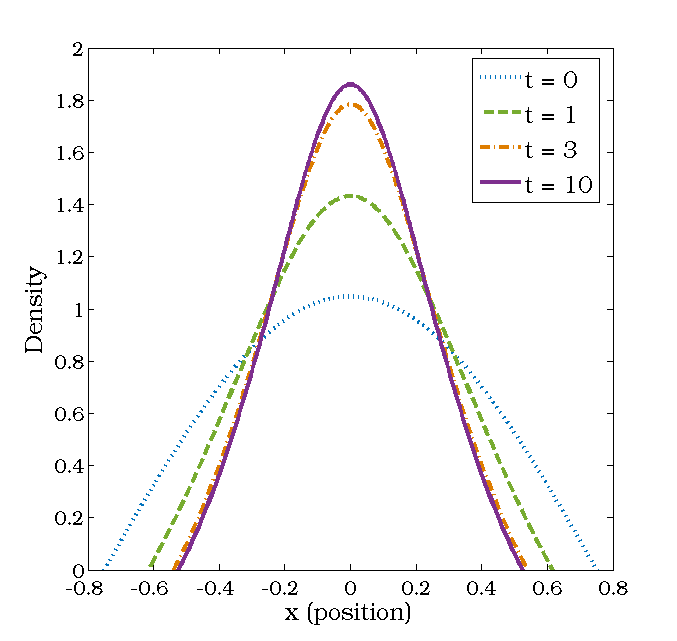}
}\subfloat[]{\protect\protect\includegraphics[scale=0.32]{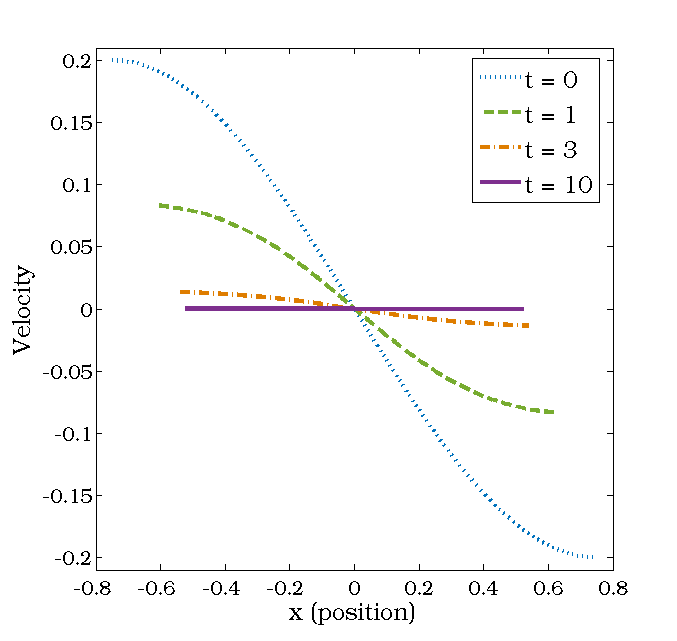}
} \newline
\subfloat[]{\protect\protect\includegraphics[scale=0.32]{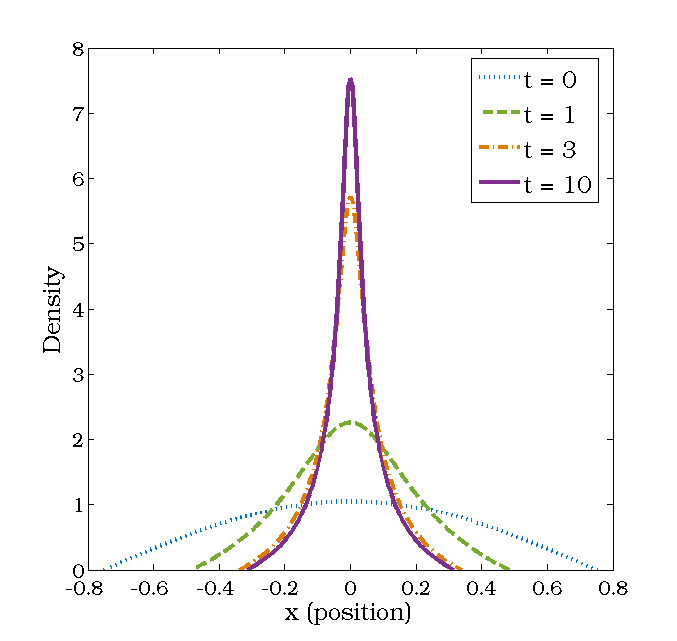}
}\subfloat[]{\protect\protect\includegraphics[scale=0.32]{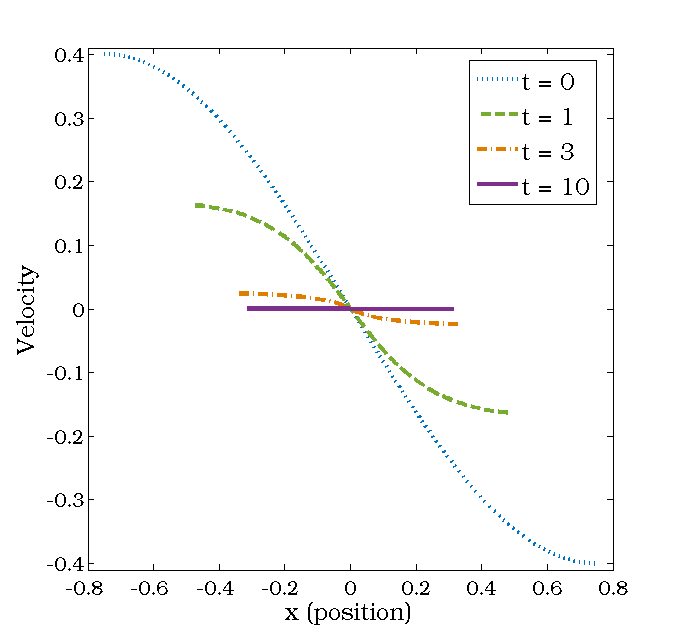}
}
\protect\protect\caption{\label{fig:eu-alig-c02} Subcritical Region: Global Existence and Asymptotic Flocking.- (a), (b): The evolution of density and velocity for the case of $c=0.2$. (c), (d): The evolution of density and velocity for the case of $c=0.4$.}
\end{figure}

When $c=0.5$, the initial data lies in the supercritical region, and thus a blow up should be expected from Theorem \ref{thm_cri}. Indeed, the numerical scheme is capable to show this blow-up phenomenon. We can observe that the derivative of the velocity becomes sharper and sharper at the origin, see the inlet in Fig. \ref{fig:eu-alig-c05}(b), while the density value at the origin gets larger and larger, Fig. \ref{fig:eu-alig-c05}(a). Actually, the simulation can not be continued after $t = 2.7311$ due to the high value of the density and the large negative derivated of the velocity at the origin. Fig. \ref{fig:eu-alig-c05} depicts the time-behavior of density until $t=2$.

\begin{figure}
\subfloat[]{\protect\protect\includegraphics[scale=0.32]{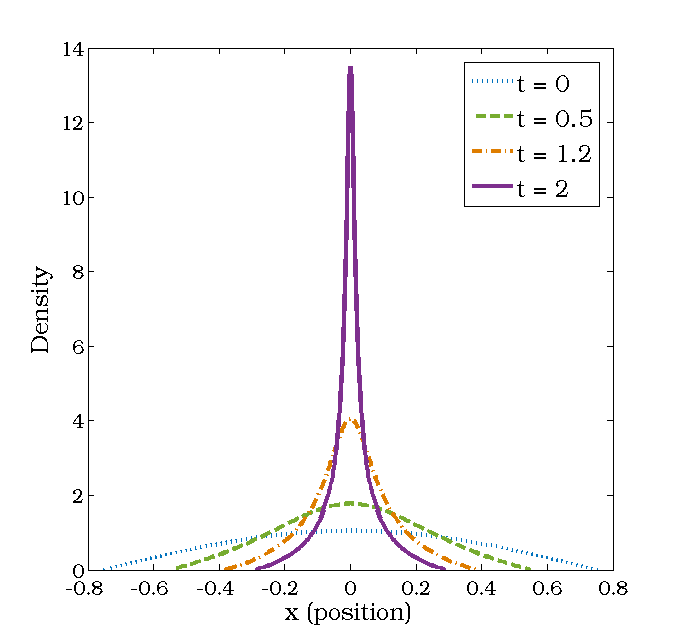}
\llap{\shortstack{%
        \includegraphics[scale=.1]{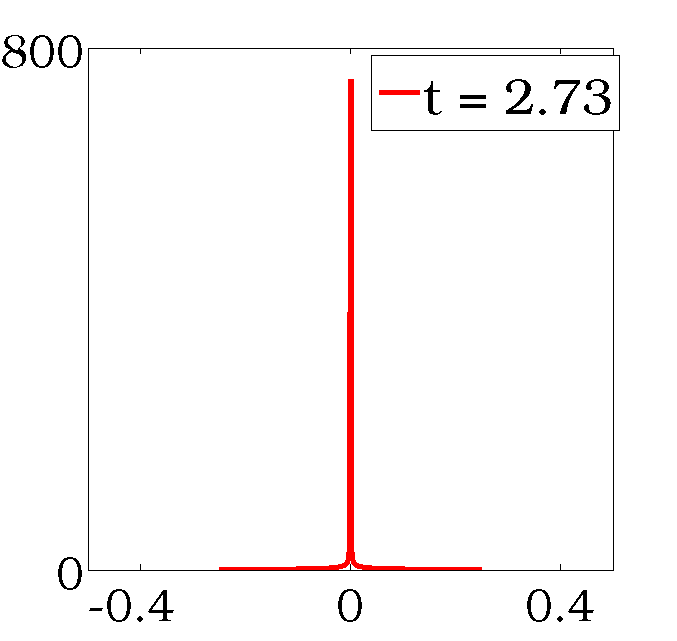}\\
        \rule{0ex}{1.25in}%
      }
  \rule{1.23in}{0ex}}
}\subfloat[]{\protect\protect\includegraphics[scale=0.32]{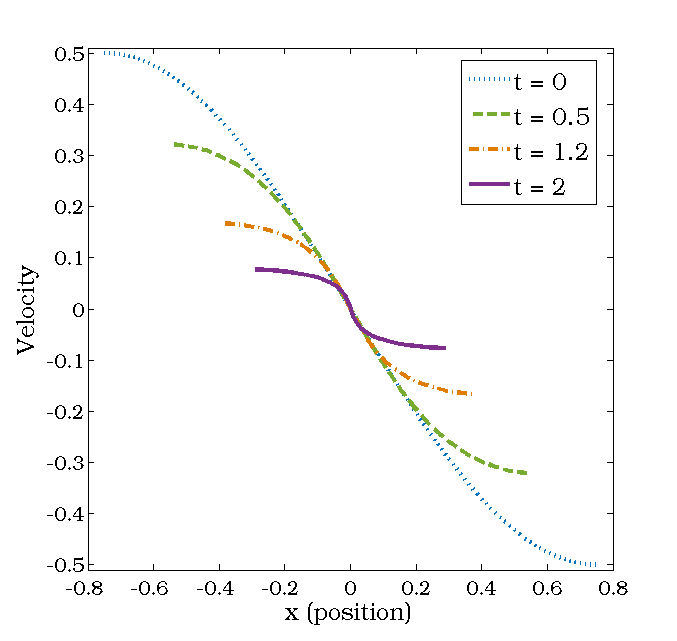}
\llap{\shortstack{%
        \includegraphics[scale=.1]{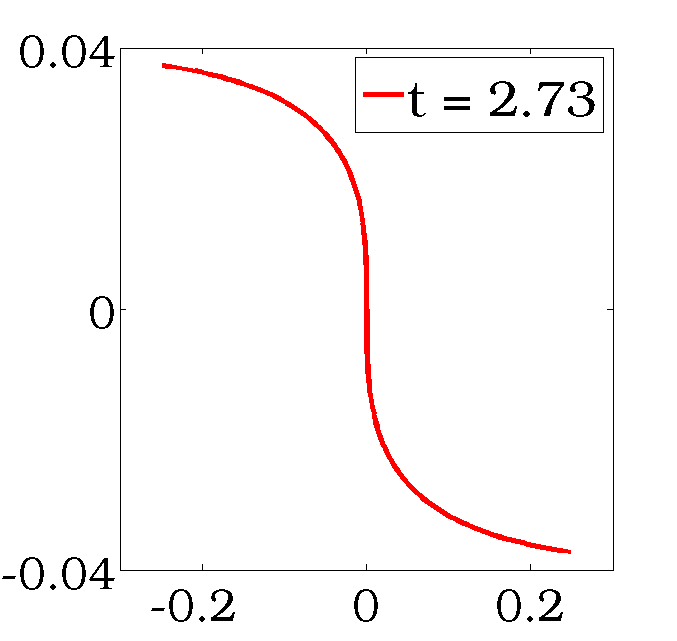}\\
        \rule{0ex}{0.23in}%
      }
  \rule{1.23in}{0ex}}
}
\protect\protect\caption{\label{fig:eu-alig-c05}Supercritical Region.- The evolution of density (a) and velocity (b) for the case of $c=0.5$. }
\end{figure}

\subsubsection{Numerical comparison of Cucker-Smale and Motsch-Tadmor equations}

The aim of this subsection is to compare the solutions of the hydrodynamic CS and MT systems. Taking a similar strategy for the CS model, the hydrodynamic MT system can be formally derived from the kinetic and the particle levels of the MT model, leading to
\begin{equation}\label{h_MT}
\begin{cases}
\partial_{t}\rho+ \pa_x \left(\rho u\right)=0,\quad x\in\R,\quad t>0,\\[2mm]
\displaystyle \pa_{t}(\rho u)+\pa_x (\rho u^2) =  \frac{1}{\psi\ast \rho}\int_{\R}\psi(x-y)(u(y)-u(x))\rho(x)\rho(y)\,dy.
\end{cases}
\end{equation}
For the system \eqref{h_MT}, the large-time behavior of solutions showing the velocity alignment is studied in \cite{MT} and the critical thresholds phenomena is provided in \cite{TT}.

For the numerical simulations, we take the parameter $\beta = 0.5$ in the communication function $\psi$. The initial conditions have been chosen following the simulations carried out at the particle level, see Fig. \ref{fig:Initial-conditions-mt} and \ref{fig:Evolution-mt-5s}. The objective has been to establish two zones in the initial domain with an important difference in mass. In addition, those two regions start with opposite velocity, so that the direction of the velocity corresponding to the largest zone will prevail.

The initial distances of Fig. \ref{fig:Initial-conditions-mt} have been preserved in the simulation, although here they have been divided by $10$ for visualization. Then, the initial positions of the nodes are
\[
\eta_i(0)=-1+\frac{7.5}{n-1}\left(i-1\right)\quad \mbox{for} \quad i=1,\cdots,n.
\]
Concerning the initial density, the sum of the mass of both regions is unit, and the relation between the individual masses is the same as in the particle level, $50/5$. Then the initial density is a piecewise function that satisfies
\[
\rho_{i}(0)=
\begin{cases}
\frac{1}{\gamma_1}\cos\left(\pi\,\dfrac{x_{i}(0)}{2}\right)\quad \mbox{if } \eta_i(0) \in [-1,1],\\[2mm]
0\quad \mbox{if } \eta_i(0) \in (1,5.5),\\[2mm]
\frac{1}{\gamma_2}\cos\left(\pi\,\dfrac{x_{i}(0)-6}{1}\right)\quad \mbox{if } \eta_i(0) \in [5.5,6.5]\\[2mm]
\end{cases}\quad \mbox{for} \quad i=1,\cdots,n,
\]
where $\gamma_1$ and $\gamma_2$ are chosen to satisfy the conditions on the masses.
With respect to the initial velocities, both groups start with opposite velocities given by
\[
u_{i}(0)=
\begin{cases}
0.1\cos\left(\pi\,\dfrac{x_{i}(0)}{2}\right)\quad \mbox{if } \eta_i(0) \in [-1,1],\\[2mm]
0\quad \mbox{if } \eta_i(0) \in (1,5.5),\\[2mm]
-0.1\cos\left(\pi\,\dfrac{x_{i}(0)-6}{1}\right)\quad \mbox{if } \eta_i(0) \in [5.5,6.5]\\[2mm]
\end{cases}\quad \mbox{for} \quad i=1,\cdots,n.
\]

In Fig. \ref{fig:mt-cs-hydro}. (a) and (b) the evolution of the density and velocity are showed, at $t=20$. As it happened in the microscopic case, for MT system the small group tends to keep the initial configuration, while for the CS case it varies more.

In order to compare the convergence rate to steady states of these systems, we consider the following quantity which measures the difference between velocities on the support of density: 
\[
R_{\rho}^{v}(t):=\sup_{x,y\in\,\Omega(t)}|u(x,t)-u(y,t)|.
\]
As depicted in Fig. \ref{fig:mt-cs-hydro}. (c), the MT model shows faster decay rate than the CS model, which is already observed at the particle level, see Fig. \ref{fig:Initial-conditions-mt}. (b). 

\begin{figure}
\subfloat[]{\protect\includegraphics[scale=0.32]{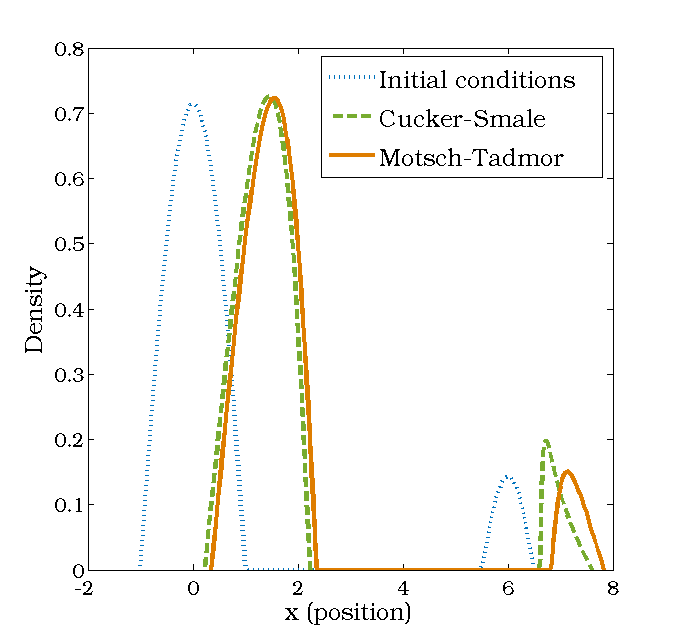}

}\subfloat[]{\protect\includegraphics[scale=0.32]{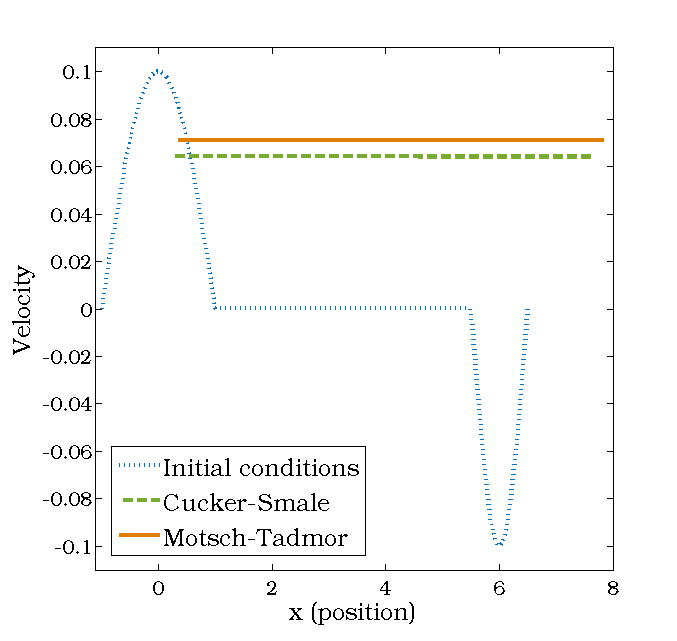}}
\newline
\centering
\subfloat[]{\protect\includegraphics[scale=0.32]{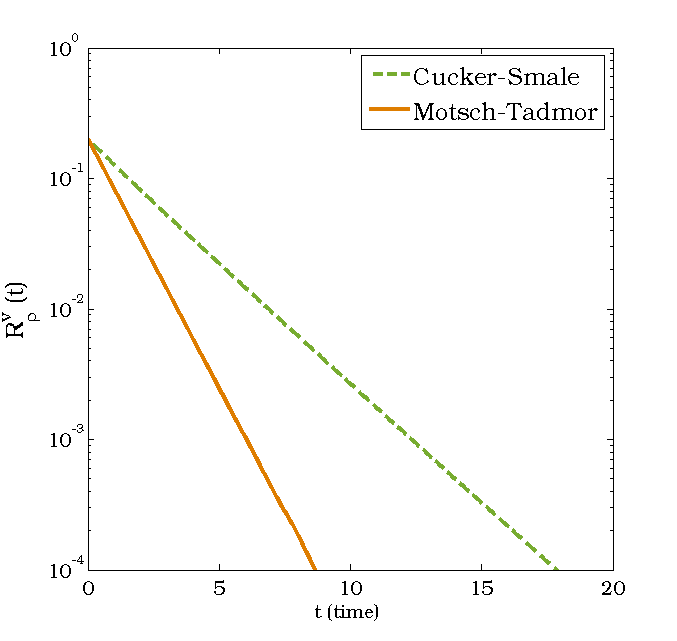}

}
\protect\caption{\label{fig:mt-cs-hydro} Comparison between CS \eqref{h_CS} with $K=0$ and MT \eqref{h_MT} systems. - (a): Evolution of the density at $t=20$. (b): Evolution of the velocity at $t=20$. (c): Decay rate of the velocity.}
\end{figure}

%
%
%
%
\subsection{Attractive-Repulsive models}

In this subsection, we consider the system \eqref{h_CS} with the following power-law potential in one dimension:
\[
K(x) = \frac{|x|^a}{a} - \frac{|x|^b}{b},
\]
with the convention that $|x|^0/0 = \log |x|$. In one dimension, $K(x) = k|x|$ corresponds to the attractive ($k>0$) or repulsive ($k<0$) Poisson force. In Section \ref{sec_EPA}, we deal with this case and study the critical thresholds analytically and numerically. In particular, there is a gap between the known subcritical and supercritical regions for the repulsive Poisson force (see Theorem \ref{thm_cri2}), we numerically analyse initial data in this region. In Section \ref{sec_EPAP}, we study the case $a= 2$ and $b = 1$. In this case, it is well-known, as discussed in Section 2.2 and at the beginning of Section 3, that there is a flocking solution
given by a density profile satisfying \eqref{eq:basic} whose solution is the characteristic function of an interval. In particular, there is an steady state with $u=0$ with this density profile. We numerically show the convergence in time of density to this steady state whenever solutions are globally defined. We also study the blow-up phenomena in this case which is not present in first-order models of swarming such as the aggregation equation \eqref{aggreg}. Finally, the case $a = 2$ and $b = 0$, i.e., when the repulsive force is more repulsive than Newtonian is also numerically explored in Section \ref{sec_EPAL}. In this case, the flocking profile/stationary state for the density satisfying \eqref{eq:basic} is given by the semi-circle law as shown in \cite{CFP} and the references therein. 

%
%
%
%

\subsubsection{Euler-Poisson-alignment system}\label{sec_EPA}
In this part, we consider the system \eqref{h_CS} with Poisson forcing term in one dimension:
\begin{equation}\label{h_CS_P}
\begin{cases}
\partial_{t}\rho+\pa_x \left(\rho u\right)=0,\quad x\in\R,\quad t>0,\\[2mm]
{\displaystyle \pa_{t}(\rho u)+ \pa_x(\rho u^2)=\int_{\R}\psi(x-y)(u(y)-u(x))\rho(x)\rho(y)\,dy - \rho \lt(\pa_x K \star \rho\rt),}
\end{cases}
\end{equation}
with $K\left(x\right)=k|x|$ attractive when $k>0$ and repulsive when $k<0$. Note that $k=\pm 0.5$ corresponds to the Newtonian interaction.
For the system \eqref{h_CS_P}, the critical thresholds are studied in \cite{CCTT} depending on the sign of $k$. The result in \cite{CCTT} is as follows.
\begin{theorem}\label{thm_cri2} Let $(\rho,u)$ be solutions to the Euler-Poisson-alignment model \eqref{h_CS_P}.

\begin{enumerate}
\item Attractive potential $\left(k>0\right)$: A unconditional finite-time
blow up of the solution will appear no matter the initial conditions. 
\item Repulsive potential $\left(k<0\right)$: As in the Euler-alignment
system, there are two different zones:
\end{enumerate}
\begin{itemize}
\item (Subcritical region) If $\partial_{x}u_{0}(x)<-\psi\star \rho_{0}(x)+\sigma_{+}(x)$
for all $x\in\mathbb{R}$, then the system has a global classical
solution. Here, $\sigma_{+}(x)=0$ whenever $\rho_{0}(x)=0$ and elsewhere
$\sigma_{+}(x)$ is the unique negative root of the equation 
\begin{equation}
\rho_{0}^{-1}(x)-\frac{1}{\psi_{M}^{2}}\left(2k+\frac{\psi_{M}\sigma_{+}(x)}{\rho_{0}(x)}-2ke^{\frac{\psi_{M}\sigma_{+}(x)}{2k \rho_{0}(x)}}\right)=0,\text{\ }\rho_{0}(x)>0.\label{eq:4.13}
\end{equation}
\item (Supercritical region) If there exists an x such that 
\[
\partial_{x}u_{0}(x)<-\psi\star \rho_{0}(x)+\sigma_{-}(x),\quad\sigma_{-}:=-\sqrt{-4k \rho_{0}(x)}.
\]
then the solution blows up in a finite time.
\end{itemize}
\end{theorem}
Note that there is a gap between the thresholds in the case of repulsive potential due to the non-locality of the velocity alignment force. If we choose the constant communication function $\psi \equiv 1$, we can close this gap, see Corollary \ref{cor_cri}.

For the initial density and velocity for the numerical simulations, we take them as in Section \ref{subsec:euler-alig}: for each node $i=1,\cdots,n$,
\[
\rho_{i}(0)=\frac{1}{\gamma}\cos\left(\pi\,\frac{x_{i}(0)}{1.5}\right) \quad \mbox{and} \quad u_{i}(0)=-c\text{\,}\sin\left(\pi\,\frac{x_{i}(0)}{1.5}\right).
\]
Similarly as before, we change the values of $c$ to consider the subcritical and supercritical regions stated in Theorem \ref{thm_cri2}.

Fig. \ref{fig:eu-po-ali-k1} shows the evolution of density and velocity for the system \eqref{h_CS_P} with $c=0.4$ and $k=0.5$. As stated in Theorem \ref{thm_cri2}, the density is blowing up at a finite time in this case. It occurs when $t=1.0788$.
\begin{figure}
\subfloat[]{\protect\protect\includegraphics[scale=0.32]{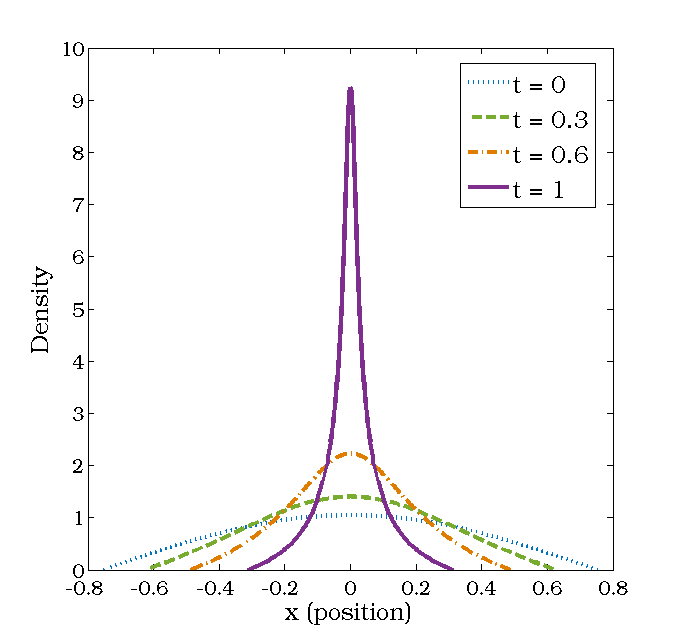}
\llap{\shortstack{%
        \includegraphics[scale=.1]{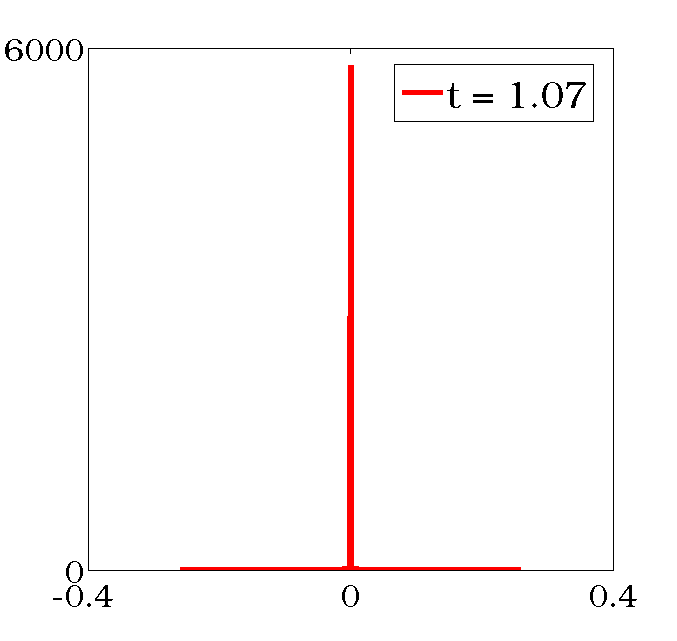}\\
        \rule{0ex}{1.25in}%
      }
  \rule{1.23in}{0ex}}
}\subfloat[]{\protect\protect\includegraphics[scale=0.32]{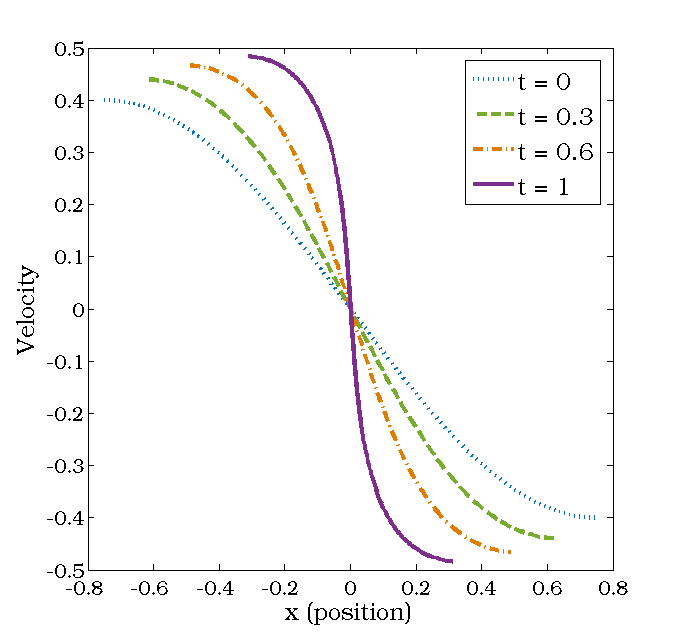}
\llap{\shortstack{%
        \includegraphics[scale=.1]{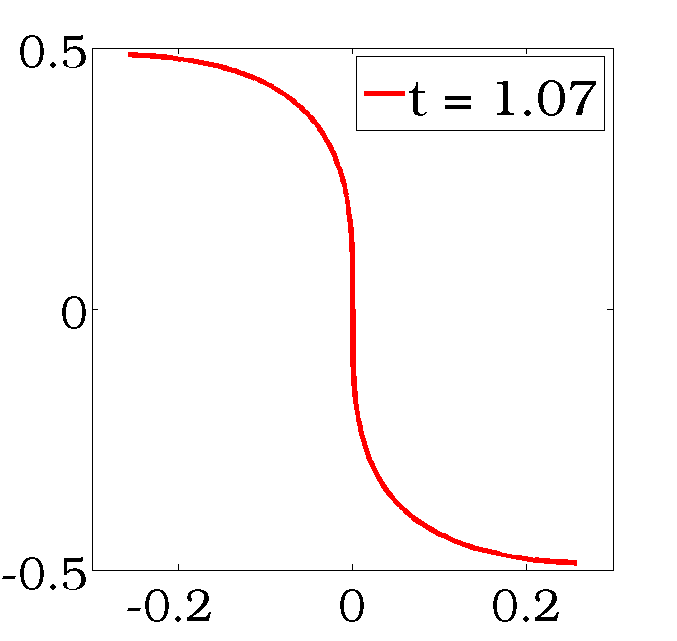}\\
        \rule{0ex}{0.23in}%
      }
  \rule{1.23in}{0ex}}
}
\protect\protect\caption{\label{fig:eu-po-ali-k1}Evolution of the density (a) and
velocity (b) for the system \eqref{h_CS_P} with $k=0.5$ and $c=0.4$.}
\end{figure}

For the repulsive potential case, we fix the value $k=-0.5$ and change the parameter $c = 0.95, 1.08, 1.2$ to consider both subcritical and supercritical regions described in Theorem \ref{thm_cri2}, see Fig. \ref{fig:eu-po.ali-IC}. To be more precise, when $c = 0.95$, $\partial_{x}u_{0}(x)+\psi\star\rho_{0}(x)$ lies in the subcritical region, for $c = 1.08$ it is between subcritical and supercritical regions. Thus it is not clear to have whether global regularity or finite-time blow-up of solutions for that case. When $c = 1.2$, it is inside the supercritical region and finite-time blowing up solution is expected from Theorem \ref{thm_cri2}. Since the initial density is independent of the parameter $c$, $\sigma_-$ and $\sigma_+$ are same for all three cases. 

A Trust-region with the Dogleg method is used in \eqref{eq:4.13} to obtain the value of $\sigma_+$. Basically it consists in a Newton-Raphson's method that applies a trust-region technique to improve robustness when starting far from the solution. Additionally, it introduces a procedure denoted as Powell dogleg. For more information about this method, see \cite{Pow}. In Matlab it could be easily implemented by the subroutine \textit{fsolve}.

\begin{figure}
\center \includegraphics[scale=0.4]{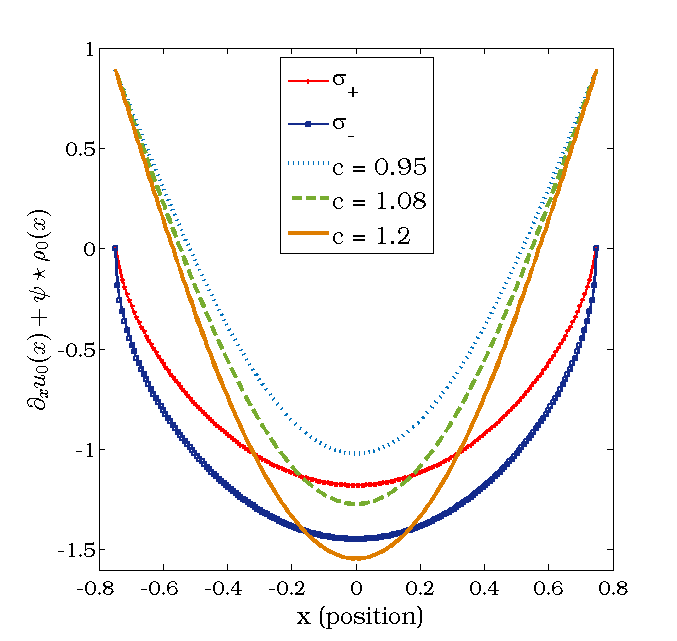}\protect\protect\caption{\label{fig:eu-po.ali-IC}Value of $\partial_{x}u_{0}(x)+\psi\star\rho_{0}(x)$
depending on $c$.}
\end{figure}

The numerical simulations of the density and the velocity for the case
$c=0.95$ at different times are shown in Fig. \ref{fig:eu-po-al-k-1-c095}. In the beginning, the density tends to be concentrated near the origin, but after some time the repulsive force changes the sign of slope of the velocity. Consequently, the density spreads and the size of support of the density is increasing. Thus there is no finite-time blow-up of solutions in this case, which is consistent with Theorem \ref{thm_cri2}. 

\begin{figure}[ht!]
\subfloat[]{\protect\protect\includegraphics[scale=0.32]{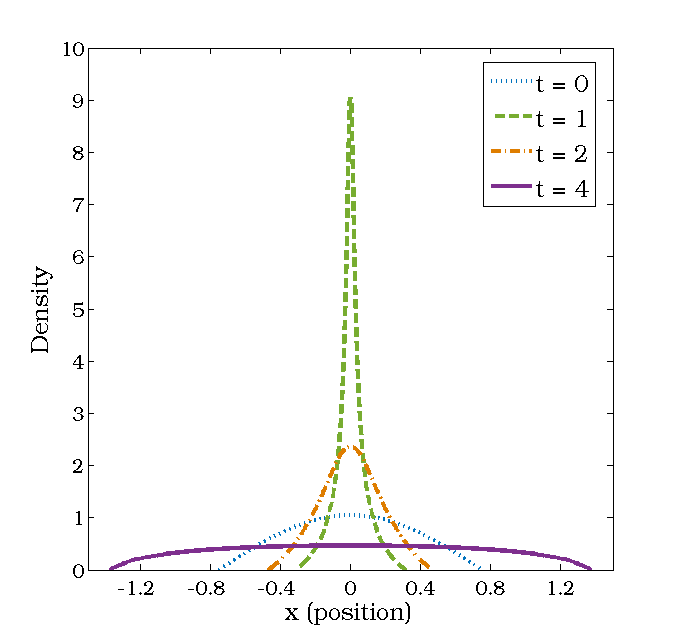}
}\subfloat[]{\protect\protect\includegraphics[scale=0.32]{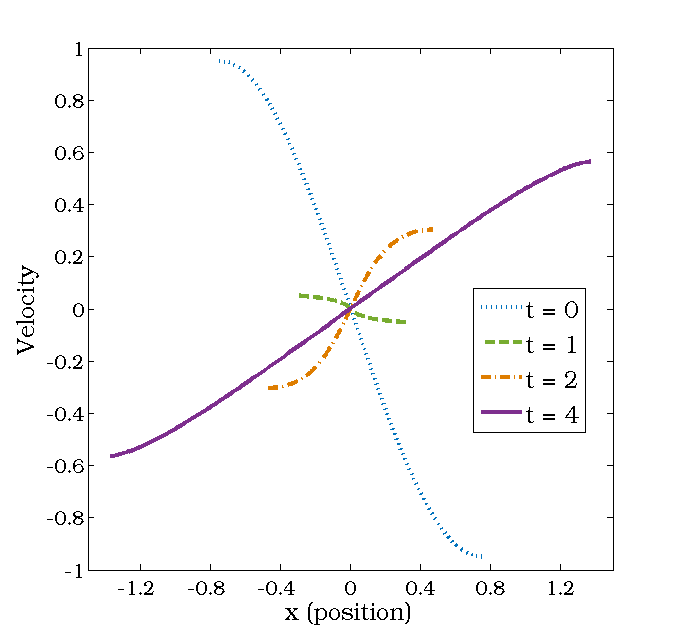}
}\protect\protect\caption{\label{fig:eu-po-al-k-1-c095}Numerical simulation of the density
(A) and velocity (B) for the case of $k=-0.5$ and $c=0.95$. }
\end{figure}

When $c=1.08$, some values of $\partial_{x}u_{0}(x)+\psi\star\rho_{0}(x)$ are between $\sigma_+(x)$ and $\sigma_-(x)$. We do not know analytically what to expect in this case. We numerically observe finite-time blow up in Fig. \ref{fig:eu-po-al-k-1-c108} after $t = 0.8107$. Our numerical exploration did not find initial data leading to global regularity when some values of $\partial_{x}u_{0}(x)+\psi\star\rho_{0}(x)$ are between $\sigma_+(x)$ and $\sigma_-(x)$.
It is an open problem to decide if dichotomy of coexistence of finite time blow-up and global existence can happen in this region.

\begin{figure}[ht!]
\subfloat[]{\protect\protect\includegraphics[scale=0.32]{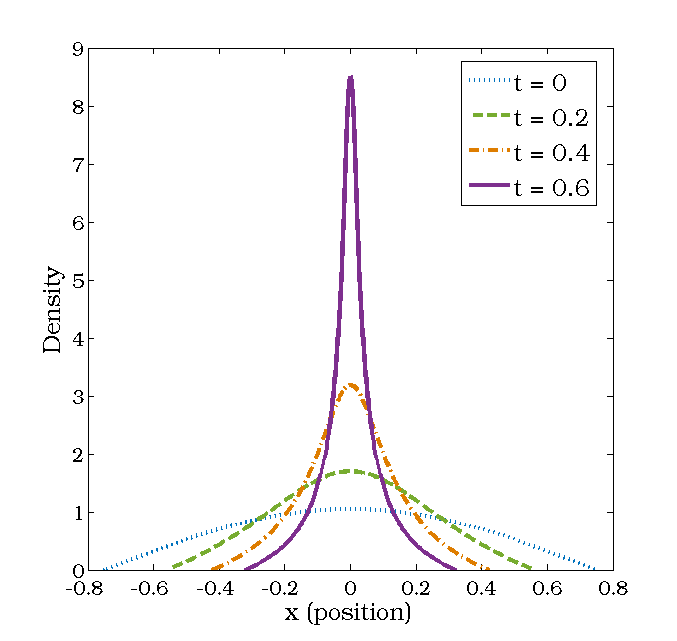}
\llap{\shortstack{%
        \includegraphics[scale=.1]{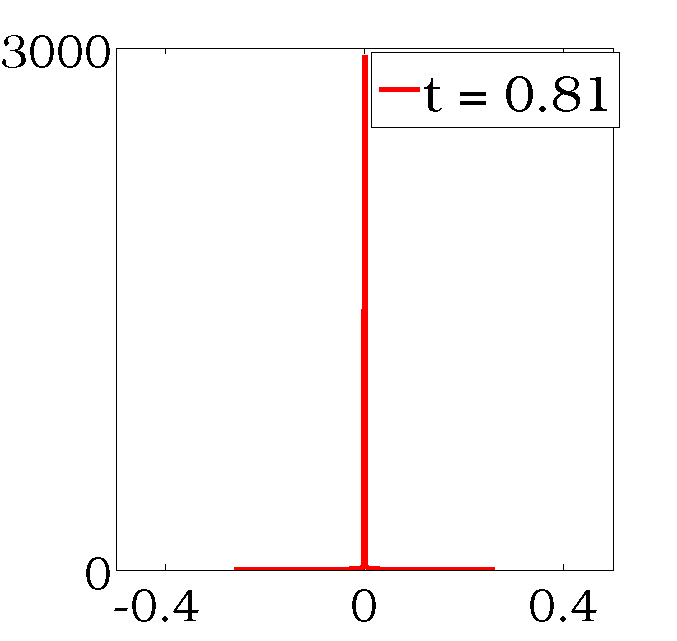}\\
        \rule{0ex}{1.25in}%
      }
  \rule{1.23in}{0ex}}
}\subfloat[]{\protect\protect\includegraphics[scale=0.32]{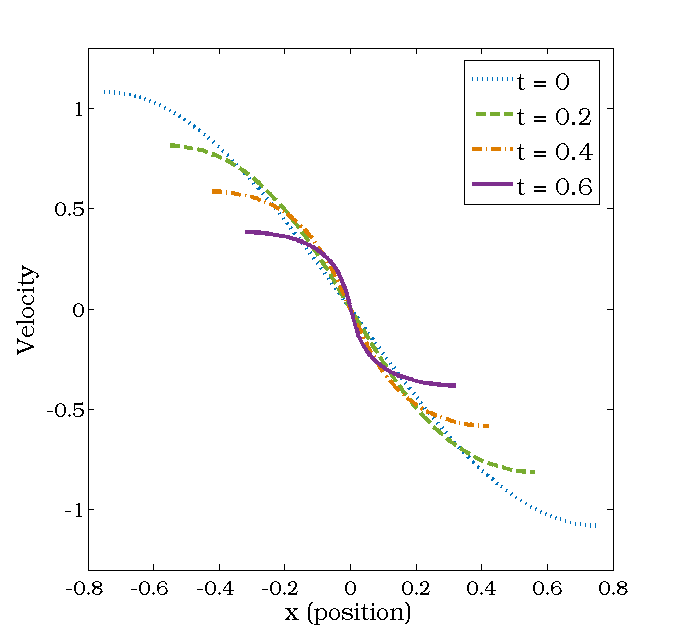}
\llap{\shortstack{%
        \includegraphics[scale=.1]{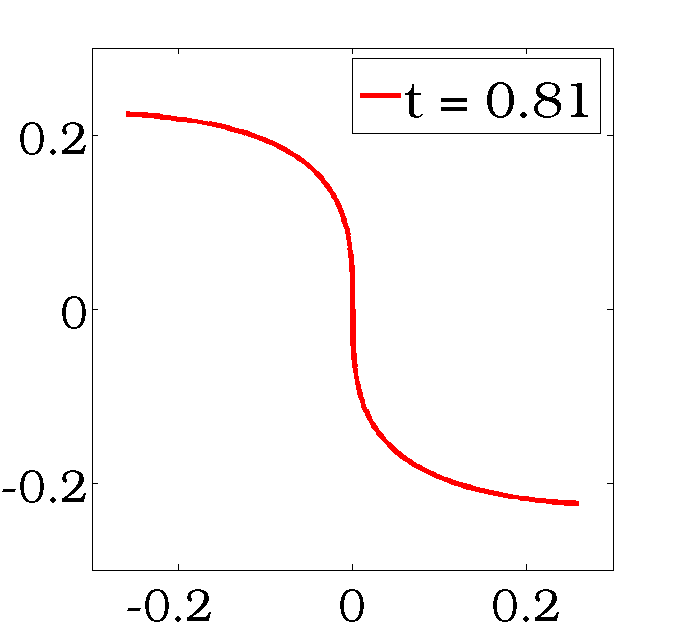}\\
        \rule{0ex}{0.23in}%
      }
  \rule{1.23in}{0ex}}
}
\protect\protect\caption{\label{fig:eu-po-al-k-1-c108}Numerical simulation of the density
(A) and velocity (B) for the case of $k=-0.5$ and $c=1.08$.}
\end{figure}

Finally, when $c=1.2$, the initial conditions are inside the supercritical region, and it is expected from Theorem \ref{thm_cri2} the blow-up of solutions in finite-time. The numerical simulations in this case show that indeed, see Fig. \ref{fig:eu-po-al-k-1-c12}. The repulsive force is not enough to prevent the blow up phenomena of the system \eqref{h_CS_P} with these initial data, and the blow-up is produced after $t = 0.6204$.

\begin{figure}[ht!]
\subfloat[]{\protect\protect\includegraphics[scale=0.32]{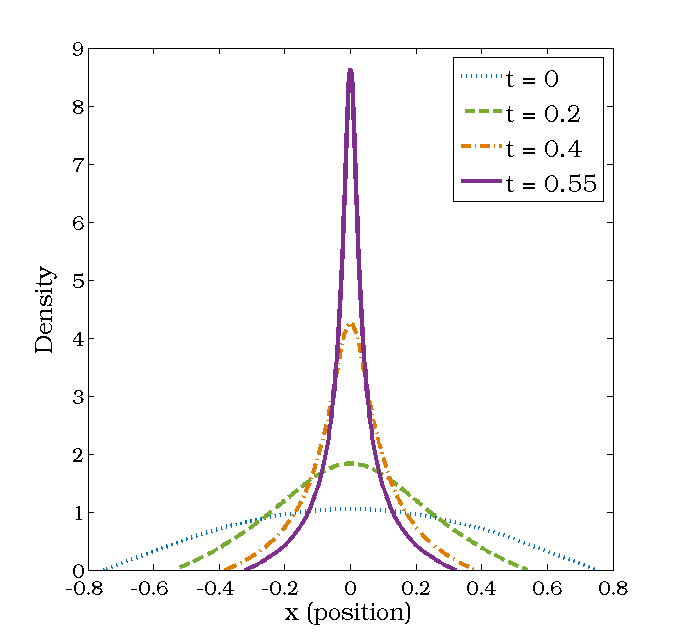}
\llap{\shortstack{%
        \includegraphics[scale=.1]{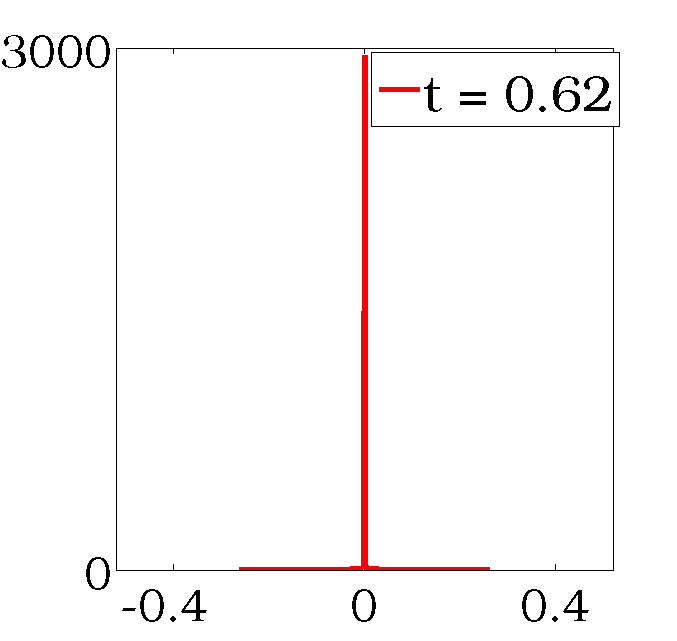}\\
        \rule{0ex}{1.25in}%
      }
  \rule{1.23in}{0ex}}
}\subfloat[]{\protect\protect\includegraphics[scale=0.32]{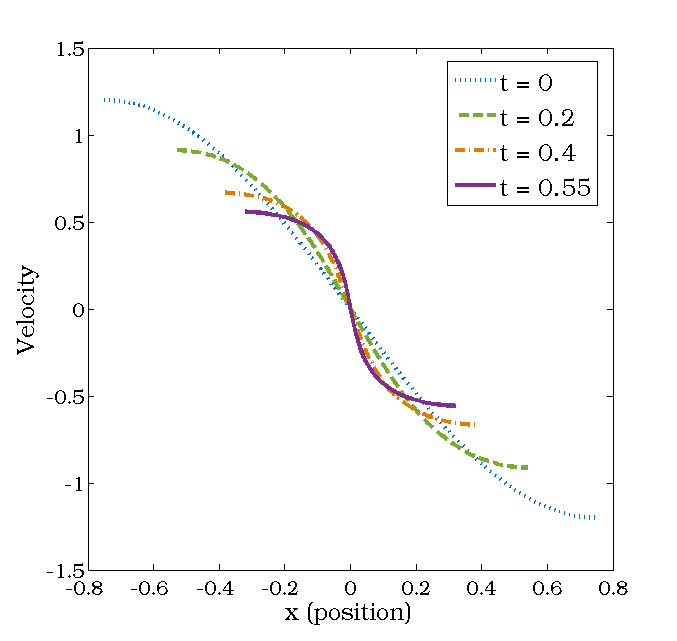}
\llap{\shortstack{%
        \includegraphics[scale=.1]{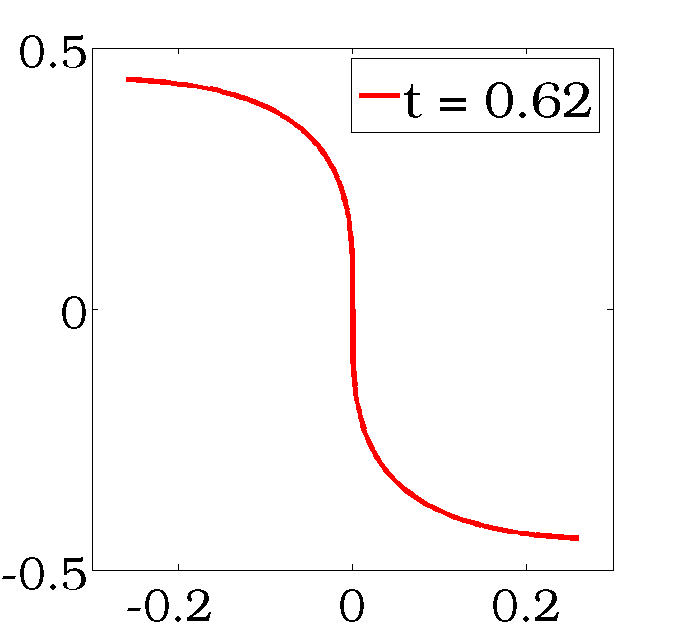}\\
        \rule{0ex}{0.23in}%
      }
  \rule{1.23in}{0ex}}
}\protect\caption{\label{fig:eu-po-al-k-1-c12}Numerical simulation of the density (A)
and velocity (B) for the case of $k=-0.5$ and $c=1.2$.}
\end{figure}

As mentioned before, the gap between thresholds occurs due to the non-locality of the velocity alignment force. In fact, if we choose the constant communication function, i.e., $\psi \equiv 1$ or $\beta = 0$, then we have a sharp critical threshold for the repulsive potential case.
\begin{corollary}\label{cor_cri} Let $(\rho,u)$ be solutions to the system \eqref{h_CS_P} with $\psi \equiv 1$ and $k < 0$.
\begin{itemize}
\item (Subcritical region) If $\partial_{x}u_{0}(x)>- \|\rho_{0}\|_{L^1} + \sigma(x)$
for all $x\in\mathbb{R}$, then the system has a global classical
solution. Here, $\sigma(x)=0$ whenever $\rho_{0}(x)=0$ and elsewhere
$\sigma(x)$ is the unique negative root of the equation 
\[
\rho_{0}^{-1}(x)-2k-\frac{\sigma(x)}{\rho_{0}(x)}+2ke^{\frac{\sigma(x)}{2k \rho_{0}(x)}}=0,\text{\ }\rho_{0}(x)>0.
\]
\item (Supercritical region) If there exists an x such that
$\partial_{x}u_{0}(x)<- \|\rho_{0}\|_{L^1}+\sigma(x)$, where the
value of $\sigma(x)$ is the one given in the $\mathit{subcritical\,region}$,
then the solution blows up in a finite time. 
\end{itemize}
\end{corollary}

This particular case has also been checked to validate our code. In fact, our code captures the dichotomy in this case quite nicely leading to similar simulations as in the case of non constant communication function. 

\subsubsection{Euler-Poisson-alignment system with power-law potential}\label{sec_EPAP}

In this part, we consider the repulsive Newtonian potential confined by a quadratic attractive potential:
\begin{equation}\label{pot_np}
K(x) = -\frac{|x|}{2} + \frac{x^2}{2}.
\end{equation}
For simplicity, we also consider a linear damping in the momentum equation instead of the nonlocal velocity alignment. In this situation, our main system reads as
\begin{equation}\label{h_np}
\begin{cases}
\partial_{t}\rho+\pa_x \left(\rho u\right)=0,\quad x\in\R,\quad t>0,\\[2mm]
{\displaystyle \pa_{t}(\rho u)+ \pa_x(\rho u^2)= -\rho u - \rho \lt(\pa_x K \star \rho\rt),}
\end{cases}
\end{equation}

If we take the constant communication function $\psi \equiv 1$ and assume the initial momentum is zero, i.e., $\int_\R \rho_0(x) u_0(x)\,dx = 0$, the system \eqref{h_np} can be derived from the system \eqref{h_CS} with the potential $K$ given in \eqref{pot_np}. 

Concerning the initial density for the numerical simulations, we again define it, for each node $i = 1,\cdots, n$, as
\[
\rho_{i}(0)=\frac{1}{\gamma}\cos\left(\pi\,\frac{x_{i}(0)}{1.7}\right),
\]
where the constant $\gamma$ is fixed so that the total mass $M_0 := \int_\R \rho_0\,dx$ has the required value. For the initial velocity, we choose
\[
u_{i}(0)=-c\text{\,}x_{i}(0)\quad\mbox{for each node } i=1,\cdots, n,
\]
where we again choose different constants $c  > 0$ to deal with the subcritical and supercritical regions.

For the system \eqref{h_np} with the potential $K$ given in \eqref{pot_np}, the sharp critical thresholds are classified in \cite{CCZ} according to the size of initial mass $M_0$. Time-asymptotic behaviors of density and velocity are also studied. The results in \cite{CCZ} for the case $M_0 < 1/4$ are as follows. 

\begin{theorem}\label{thm_crinew}
Let $(\rho,u)$ be a classical solution to the system \eqref{h_np} with the potential \eqref{pot_np}. Suppose that the initial density is compactly supported with the initial mass satisfying $M_0 <1/4$. Then the solution blows up in finite time if and only if there exists a $x^* \in \Omega_0 :=$ supp$(\rho_0)$ such that
\[
\partial_x u_0(x^*)<0,\quad M_0 - \rho_0(x^*)<\lambda_1\partial_x u_0 (x^*)
\]
and
\[
\rho_0(x^*)\leq \left(\lambda_1 \partial_x u_0 (x^*) + \rho_0(x^*)-M_0\right)^{-\lambda_2/\sqrt{\Xi}} \left(\lambda_2 \partial_x u_0 (x^*) + \rho_0(x^*)-M_0\right)^{\lambda_1\sqrt{\Xi}}.
\]
Here the constants $\lambda_i<0, i=1,2$ and $\Xi > 0$ are given as
\[
\lambda_1 \coloneqq \dfrac{-1+\sqrt{1-4M_0}}{2},\quad\lambda_2 \coloneqq \dfrac{-1-\sqrt{1-4M_0}}{2},\quad \textrm{and}\quad \Xi := 1-4M_0.
\]
Furthermore, if there is no finite-time blow-up, we have
\[
\rho(x,t) \to M_0 \mathbf{1}_{[a,b]} \quad \mbox{in } L^1 \quad \mbox{and} \quad u(x,t) \to 0, \quad \mbox{in } L^\infty \quad \mbox{as} \quad t \to \infty,
\]
exponentially fast, where $a,b$ are constants given by
$$\begin{aligned}
b &=  \frac{1}{M_0}\lt(\int_\R x \rho_0(x)\,dx + \int_\R (\rho_0 u_0)(x)\,dx \rt) + \frac12,\cr
a &=  \frac{1}{M_0}\lt(\int_\R x \rho_0(x)\,dx + \int_\R (\rho_0 u_0)(x)\,dx \rt) - \frac12.
\end{aligned}$$
\end{theorem}
We refer to \cite{CCZ} for the sharp critical thresholds and the large-time behavior for global-in-time solutions for $M_0 \geq 1/4$.

In order to check numerically the critical thresholds stated in Theorem \ref{thm_crinew}, two numerical simulations are carried out in Fig. \ref{fig:new_critthres}. First, the mass is set to be $M_0=0.2$, and then two cases $c = 0.9, 1.1$ are considered. Note that our initial conditions for $\rho_0$ and $u_0$ imply 
\[
\int_\R x \rho_0(x)\,dx = \int_\R (\rho_0 u_0)(x)\,dx = 0.
\]
When $c=0.9$, the initial data lie in the subs critical region, i.e., the initial data do not satisfy the conditions in Theorem \ref{thm_crinew}, and subsequently, the density and the velocity converge to $M_0 \mathbf{1}_{[-1/2,1/2]}$ and $0$ as time goes on, respectively.  On the other hand, for the case $c=1.1$, there is a finite time blow up caused by an infinite slope of the velocity on the boundary. No numerical simulation has been provided for the case of $M_0 \geq 1/4$ because the critical threshold established in \cite{CCZ} for that case involves a more involved requirement on the initial conditions. However, the dichotomy of behaviors obtained is similar.
\begin{figure}
\subfloat[]{
\protect\includegraphics[scale=0.32]{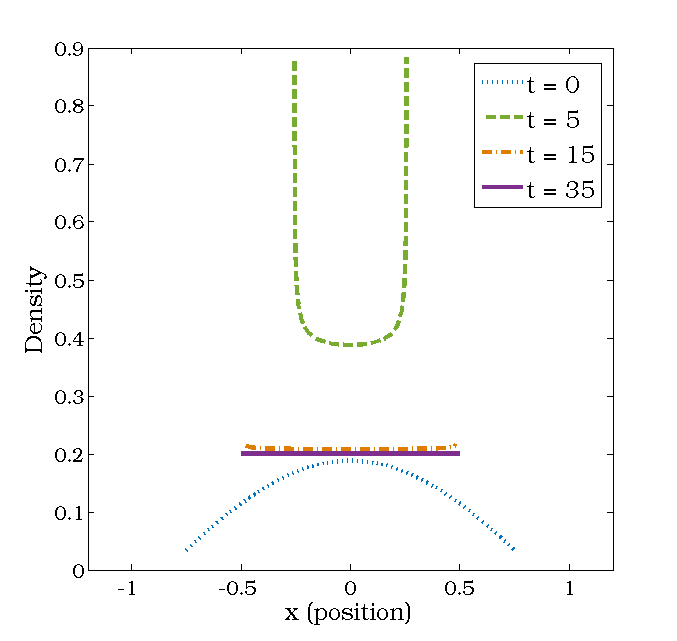}
\llap{\shortstack{%
        \includegraphics[scale=.093]{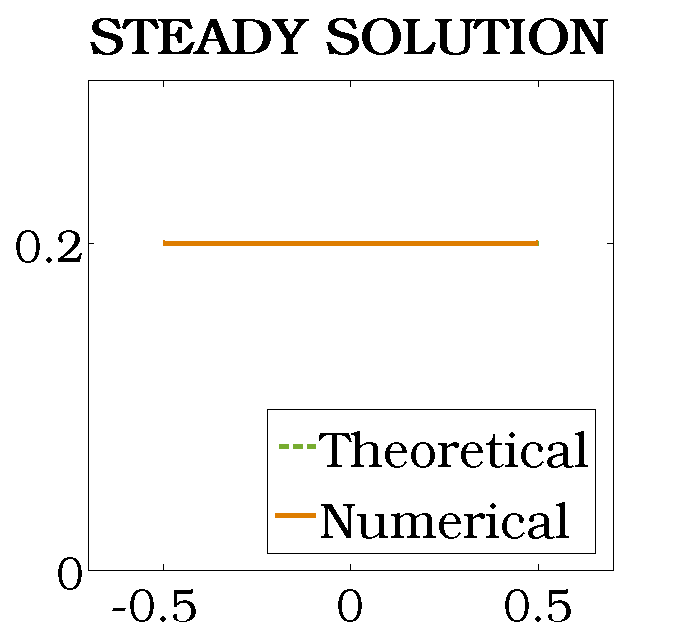}\\
        \rule{0ex}{1.28in}%
      }
  \rule{1.29in}{0ex}}
}\subfloat[]{\protect\includegraphics[scale=0.32]{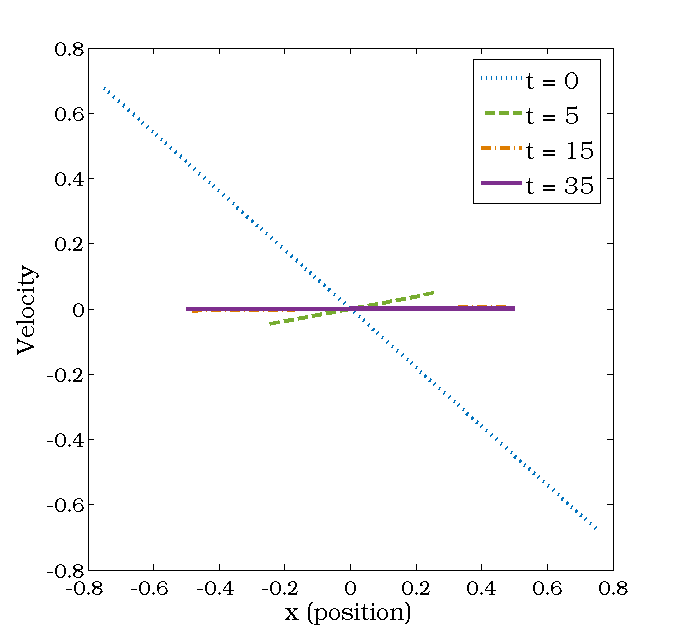}
}\newline
\subfloat[]{
\protect\includegraphics[scale=0.32]{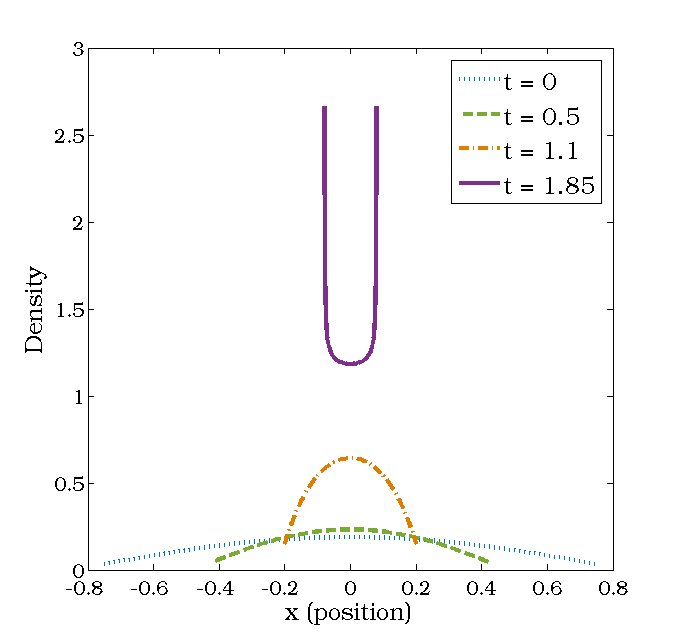}
}\subfloat[]{\protect\includegraphics[scale=0.32]{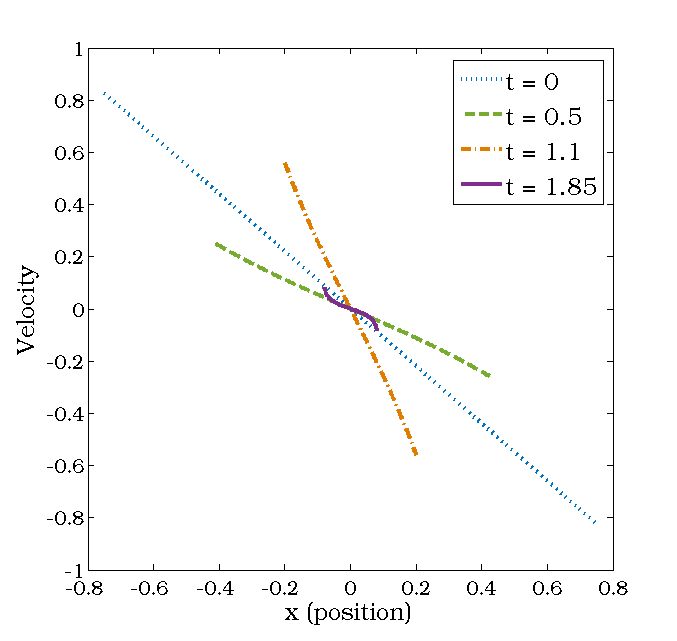}
\llap{\shortstack{%
        \includegraphics[scale=.1]{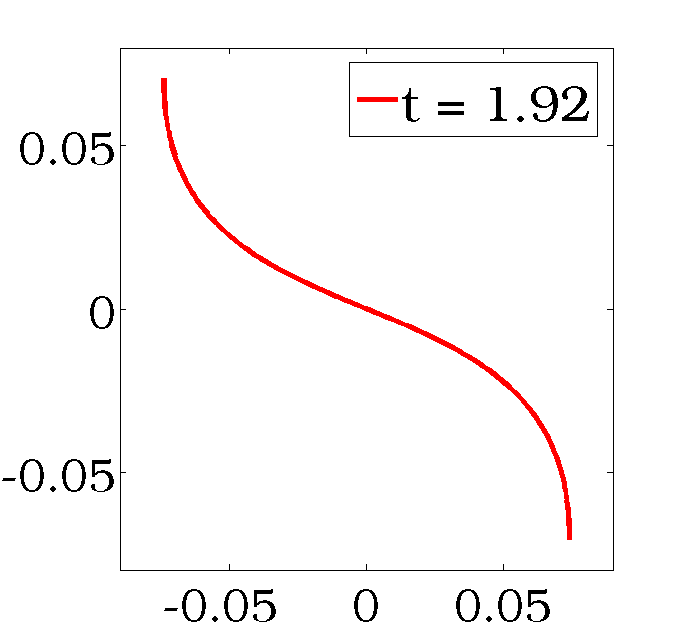}\\
        \rule{0ex}{0.23in}%
      }
  \rule{1.23in}{0ex}}
}\protect\caption{\label{fig:new_critthres} Numerical simulations with respect to Theorem \ref{thm_crinew}.- (a), (b): Time behavior of the density and the velocity for the case $c=0.9$. (c), (d): Time behavior of the density and the velocity for the case $c=1.1$.}
\end{figure}

By employing the argument proposed in \cite{CCTT}, we provide an estimate on the blow-up time for the system \eqref{h_np} with the potential \eqref{pot_np}. It is worth mentioning that this blow-up analysis does not depend on the size of mass. Differentiating the momentum equation of system \eqref{h_np} with respect to $x$, we can rewrite it as
\begin{equation}\label{h_np_blowup}
\begin{cases}
\rho' = - \rho d\\[2mm]
d' = -d^2 - d + 2\rho - M_0
\end{cases}
\end{equation}
where $\{ \}'$ denotes the time derivative along the characteristic flow $\eta$ defined in \eqref{eq:char} and $d := \pa_x u$.
\begin{theorem}\label{thm_blowupnew}
Let $(\rho,u)$ be a classical solution to the system \eqref{h_np} with the potential \eqref{pot_np} on the time interval $[0,T]$. Suppose that there exists an x such that 
\[
d_0(x) < 0 \quad \mbox{and}  \quad \frac{1 + d_0(x)}{\rho_0(x)} +2\log\left(1 - \frac{d_0(x)}{2\rho_0(x)}\right) \leq 0.
\]
Then the life span $T$ of the solution should be finite time and moreover
\[
T \leq 2\inf_{x \in \mathcal{S}}\log\left(1 - \frac{d_0(x)}{2\rho_0(x)}\right),
\]
where $\mathcal{S}$ is defined as
\[
\mathcal{S}:= \lt\{x \in \R\,:\, d_0(x) < 0 \quad \mbox{and}  \quad \frac{1 + d_0(x)}{\rho_0(x)} +2\log\left(1 - \frac{d_0(x)}{2\rho_0(x)}\right) \leq 0\rt\}.
\]
\end{theorem}
\begin{proof} Set $\beta = d/\rho$. Then it follows from \eqref{h_np_blowup} that $\beta$ satisfies
\[
\beta' = \frac{1}{\rho^2}\lt(d' \rho - d\rho' \rt) = \frac{1}{\rho^2}\lt( -d\rho + 2\rho^2 - \rho M_0\rt) = -\beta + 2 - \frac{M_0}{\rho} \leq -\beta + 2.
\]
This yields
\[
\beta \leq 2 + (\beta_0 - 2)e^{-t} \quad \mbox{for} \quad t \geq 0.
\]
On the other hand, it follows from the continuity equation that
\[
\rho' = -\rho^2 \beta, \quad \mbox{i.e.,} \quad \rho^{-1} = \rho_0^{-1} + \int_0^t \beta(s)\,ds \leq \rho_0^{-1} + 2t + (\beta_0 - 2)(1 - e^{-t}).
\]
Set $f(t) := \rho_0^{-1} + 2t + (\beta_0 - 2)(1 - e^{-t})$, then 
\[
f_0 = \rho_0^{-1}>0 \quad \mbox{and} \quad \lim_{t \to +\infty}f(t) = \infty.
\]
On the other hand, $f'(t) =  2 + (\beta_0 - 2)e^{-t}$, thus if there exists a $t_* > 0$ such that $f'(t_*) = 0$ and $f(t_*) \leq 0$, then the density $\rho$ is blowing up until this time $t_* > 0$. Note that $f'(t_*) = 0$ implies $e^{-t_*} = 2/(2-\beta_0)$. This yields $\beta_0 < 0$, i.e., $d_0 < 0$ due to $e^{-t_*} \in (0,1)$. Then for $d_0 < 0$ we get
\[
f(t_*) = \rho_0^{-1} + \beta_0 + 2t_* = \rho_0^{-1} +\beta_0 +2\log\left(\frac{2-\beta_0}{2}\right).
\]
Hence if there exists a $x$ such that
\[
d_0(x) < 0 \quad \mbox{and}  \quad \rho_0^{-1}(x) +\beta_0(x) +2\log\left(\frac{2-\beta_0(x)}{2}\right) \leq 0,
\]
then the life-span $T$ of the solution $(\rho,u)$ should be finite. Furthermore, the time $T$ satisfies 
\[
T \leq 2\inf_{x \in \mathcal{S}}\log \lt(\frac{2-\beta_0(x)}{2}\rt).
\]
\end{proof}

A study concerning the qualitative properties of the dynamics of the system \eqref{h_np} with the potential \eqref{pot_np} are also conducted. Depending on the initial conditions, the density and position may converge to the steady state with or without oscillations around that state. In Fig. \ref{fig:new_comparison}, it is depicted how the density and position of the boundary nodes evolves depending on the initial mass or the initial velocity. Generally, it shows in Fig. \ref{fig:new_comparison} (a) and (c) that there is a limit mass below in which there are no oscillations. However, this limit mass may change with the initial velocity. With respect to the influence of the initial velocity for a fixed mass according to  Fig. \ref{fig:new_comparison} (b) and (d), it is possible to deduce that the more negative the initial slope of the velocity is, the larger the tendency to the oscilations is. 

\begin{figure}[ht!]
\subfloat[]{
\protect\includegraphics[scale=0.32]{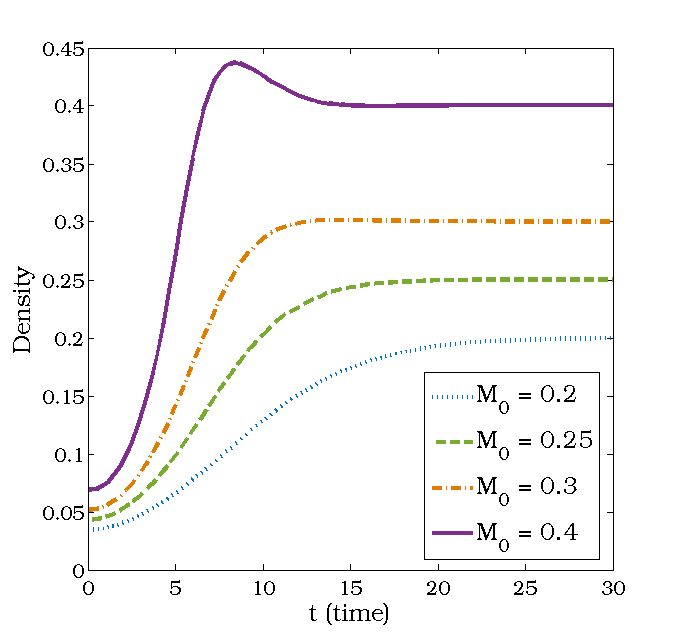}
}\subfloat[]{\protect\includegraphics[scale=0.32]{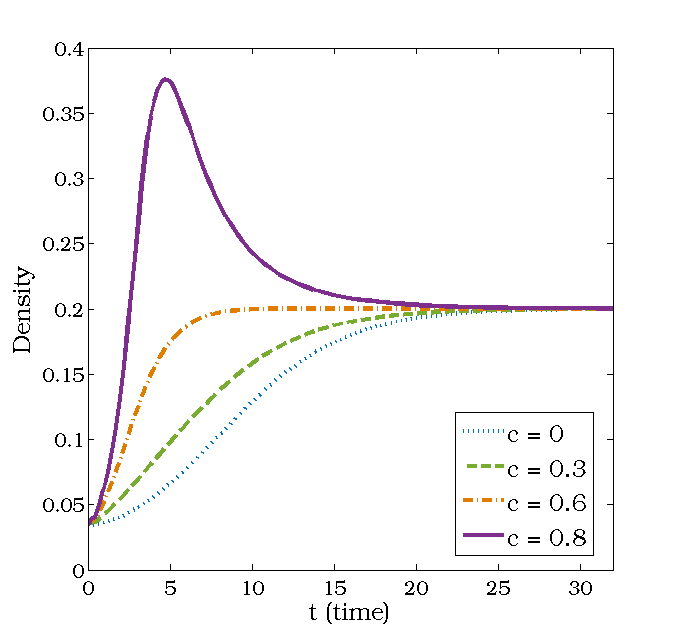}
}\newline
\subfloat[]{
\protect\includegraphics[scale=0.32]{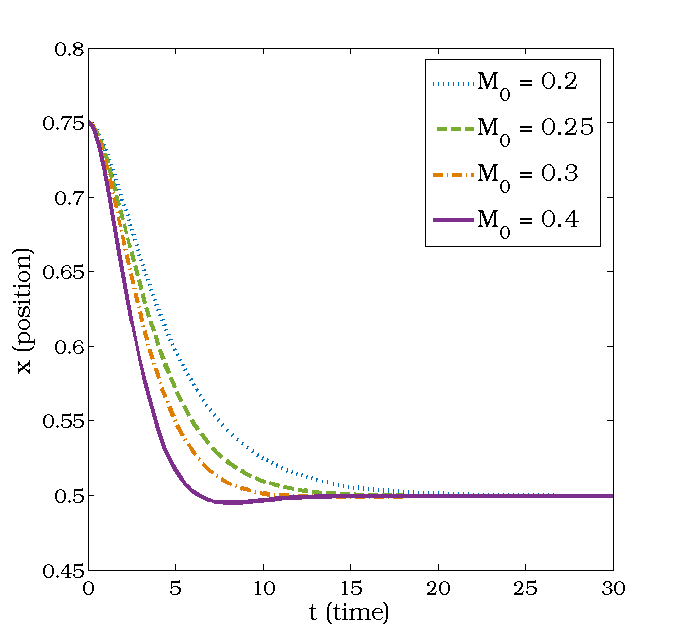}
}\subfloat[]{\protect\includegraphics[scale=0.32]{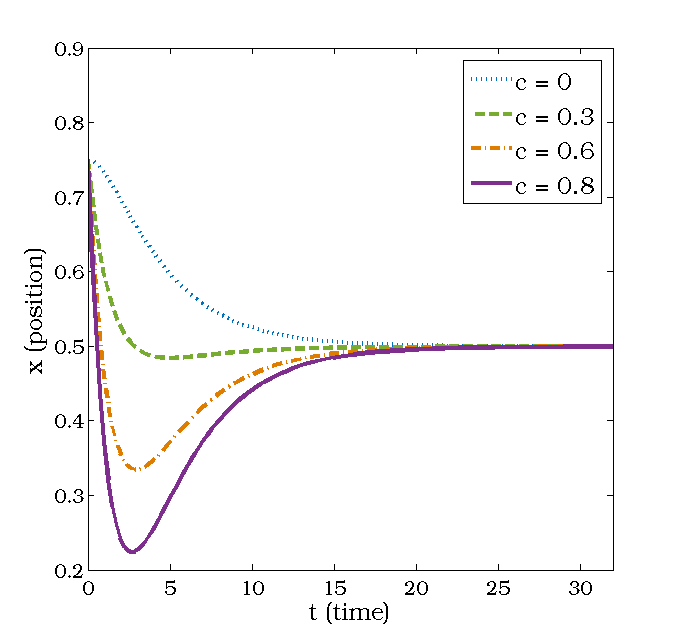}
}\protect\caption{\label{fig:new_comparison} Dynamical behavior of the covergence of boundary nodes.- (a): Density with different values of the mass $M_0$ and $c=0$. (b): Density with different values of the parameter $c$ and $M_0=0.2$. (c): Position with different values of the mass $M_0$ and $c=0$. (d): Position with different values of the parameter $c$ and $M_0=0.2$.}
\end{figure}

Finally, some numerical simulations with the CS nonlocal velocity alignment instead of linear damping are conducted. Although no theoretical result is known for that case, our numerical simulations demonstrate that the total mass of the system and the initial conditions affect the global behavior of the solution in similar way as in the linear damping case.
\begin{figure}[ht!]
\subfloat[]{
\protect\includegraphics[scale=0.32]{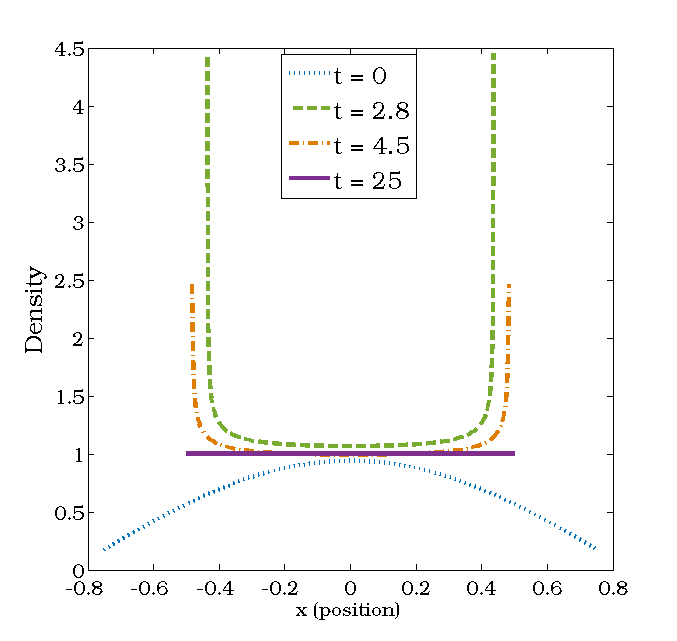}
\llap{\shortstack{%
        \includegraphics[scale=.1]{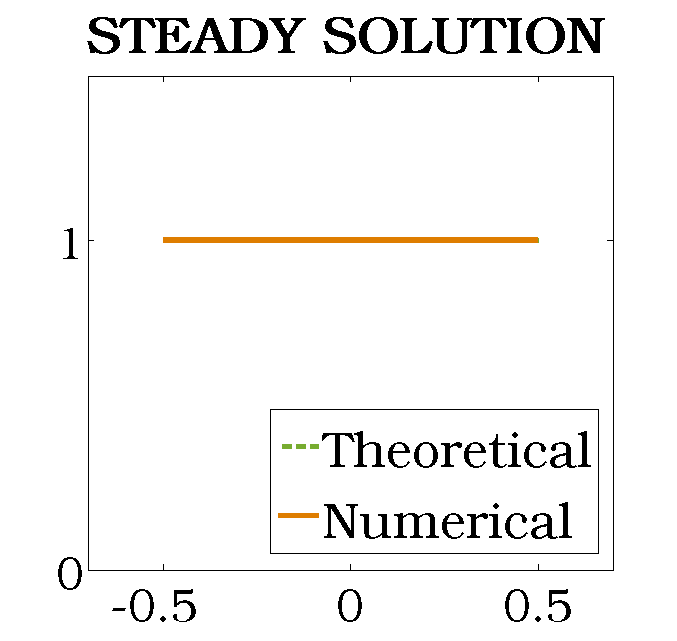}\\
        \rule{0ex}{0.7in}
      }
  \rule{0.75in}{0ex}}

}\subfloat[]{\protect\includegraphics[scale=0.32]{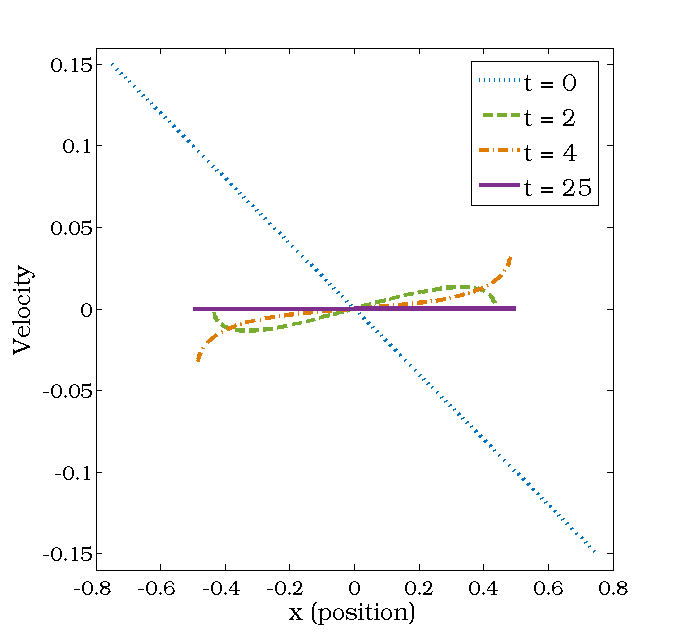}
}\newline
\subfloat[]{
\protect\includegraphics[scale=0.32]{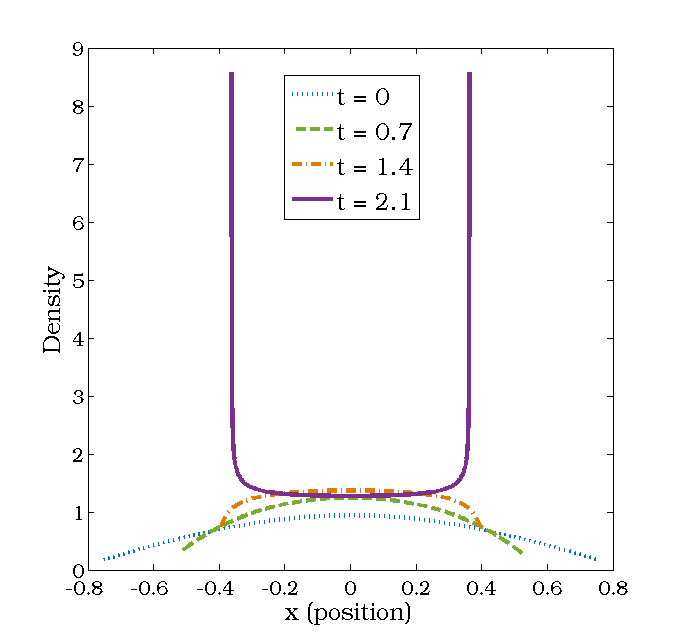}
}\subfloat[]{\protect\includegraphics[scale=0.32]{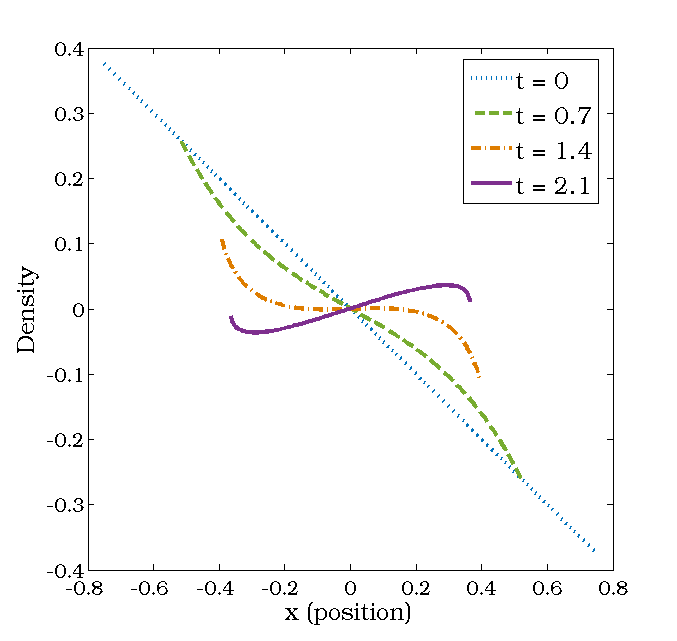}
\llap{\shortstack{%
        \includegraphics[scale=.1]{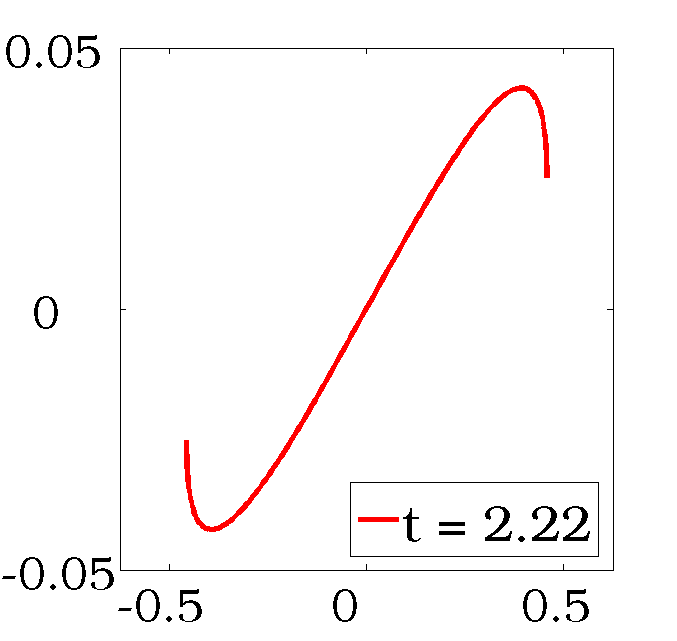}\\
        \rule{0ex}{0.23in}%
      }
  \rule{1.23in}{0ex}}

}\protect\caption{\label{fig:new_cs} Numerical simulations with CS nonlocal velocity alignment.- (a), (b): Time behavior of the density and the velocity for the case $M_0=1$ and $c=0.2$. (c), (d): Time behavior of the density and the velocity for the case $M_0=1$ and $c=0.2$.}
\end{figure}

In Fig. \ref{fig:new_cs}, two different simulations are depicted, and they differ in the initial velocity. The first one, Fig. \ref{fig:new_cs} (a), corresponds to $M_0=1$ and $c=0.2$, and we observe that there is no blow up in finite time reaching eventually the steady state. In contrast, the other simulation, increasing the value of the parameter c up to $0.5$, shows finite-time blow-up due to the infinite slope of the velocity on the boundary. The numerical time of blow up is $t=2.22$.

%
%
%
%

\subsubsection{Stronger repulsive potential than Newtonian}\label{sec_EPAL}

In this part, we deal with a potential more repulsive than Newtonian at the origin confined by the quadratic attractive potential given by
\begin{equation}\label{pot_nnp}
K(x) = -\log|x| + \frac{x^2}{2}.
\end{equation}
For simplicity, we again consider the linear damping in the momentum equation. Steady solutions for this problem corresponds again to find density profiles $\rho$ satisfying \eqref{eq:basic} together with $u=0$ on the support of the density $\rho$. They can be characterized as the global minimizer of certain energy functional and they can be computed explicitly, see \cite{CFP} and \cite{ST}. In fact, they are given by the semicircle law, that is, their graph is a semicircle.

We choose as initial density for the numerical simulations the positive part of a cosine function. For each node $i \in \{1,\cdots, N\}$, it is given by
\[
\rho_{i}(0)=\frac{1}{\gamma}\cos\left(\pi\,\frac{x_{i}(0)}{1.5}\right),
\]
where the constant $\gamma$ is computed so that the total mass $M_0 := \int_\R \rho_0\,dx$ has the required value. With respect to the initial velocity, it is chose to be
\[
u_{i}(0)=-c\text{\,}x_{i}(0)\quad \mbox{for } \:i=1,...,n,
\]
where the constant $c\in\mathbb{R}^{+}$ will be varied to study different
initial conditions in the simulations.

We first find a blow up estimate for the system \eqref{h_np} with the potential \eqref{pot_nnp}. By differentiating the momentum equation of system \eqref{h_np}, it is obtained that 
\begin{align}\label{gd_est}
\begin{aligned}
&\rho' = -\rho d,\cr
&d' = - \left(d^2 + d +  \int_{\Omega(t)} \frac{\rho(y)}{|x-y|^2}\,dy+ M_0\right) \leq - (d^2 + d + M_0)  ,
\end{aligned}
\end{align}
\noindent and for the case of $1 - 4M_0 \geq 0 $ it is found that 
\begin{align}\label{gd_est}
\begin{aligned}
&\rho' = -\rho d,\cr
&d' \leq -(d - d_-)(d - d_+),
\end{aligned}
\end{align}
where
\[
d_\pm := \frac{-1 \pm \sqrt{1 - 4M_0}}{2}.
\]

\begin{theorem}\label{thm_blowupnonnew}Let $(\rho,u)$ be a classical solution to the system \eqref{h_np} with the potential \eqref{pot_nnp} on the time interval $[0,T]$. Suppose that $1 - 4M_0 \geq 0$. Then the life span $T$ of the solution $(\rho,u)$ should be finite if there exists a x such that $\partial_x u_0(x)<\frac{-1 - \sqrt{1 - 4M_0}}{2}$. Moreover, 
\[
T \leq \frac{1}{d_- - d_0}.
\]
\end{theorem}

\begin{proof}We divide the proof into two steps.

{\bf Step 1.-} If $d_0 < d_-$, then $d(t) < d_-$ for $t \in [0,T]$. 
Set $$\mathcal{T}:= \lt\{t \in [0,T] : d(s) < d_- \mbox{ for } s \in [0,t] \rt\}.$$ Then $\mathcal{T}$ is not empty since $0 \in \mathcal{T}$. Furthermore if we set $\mathcal{T}^* = \sup \mathcal{T}$, then $\mathcal{T}^* > 0$ since $d(t)$ is continuous. Suppose that $\mathcal{T}^* < T$, then we get $\lim_{t \to \mathcal{T^*}-} d(t) = d_-$. On the other hand, it follows from \eqref{gd_est} that
\[
d' \leq -(d - d_-)(d - d_+) \quad \mbox{for} \quad t \in [0,\mathcal{T}^*),
\]
and this yields $d'(t) < 0$ for $t \in [0,\mathcal{T}^*)$ since $d(t) < d_- < d_+$ for $t \in [0,\mathcal{T}^*)$. This is a contradiction to
\[
d_- = \lim_{t \to \mathcal{T^*}-} d(t) \leq d_0 < d_-.
\]
Thus $\mathcal{T}^* \geq T$ and $d(t) < d_-$ for $t \in [0,T]$. 

{\bf Step 2.-} If $d_0 < d_-$, then the life-span of smooth solutions $T$ should be finite. \newline
Since $d(t) < d_- < d_+$ for $t \in [0,T]$, we get
\[
d' \leq -(d - d_-)(d - d_+) \leq -(d - d_-)^2, \quad \mbox{i.e.,} \quad (d - d_-)'\leq -(d - d_-)^2.
\]
This implies
\[
d(t) \leq \frac{1}{(d_0 - d_-)^{-1} + t} + d_-.
\]
Since $d_0 < d_-$, the right hand side of the above equality diverges to $- \infty$ when $t \rightarrow (d_- - d_0)^{-1}$. This concludes that the life-span $T$ should be less than $(d_- - d_0)^{-1}$.
\end{proof}

\begin{figure}
\subfloat[]{
\protect\includegraphics[scale=0.32]{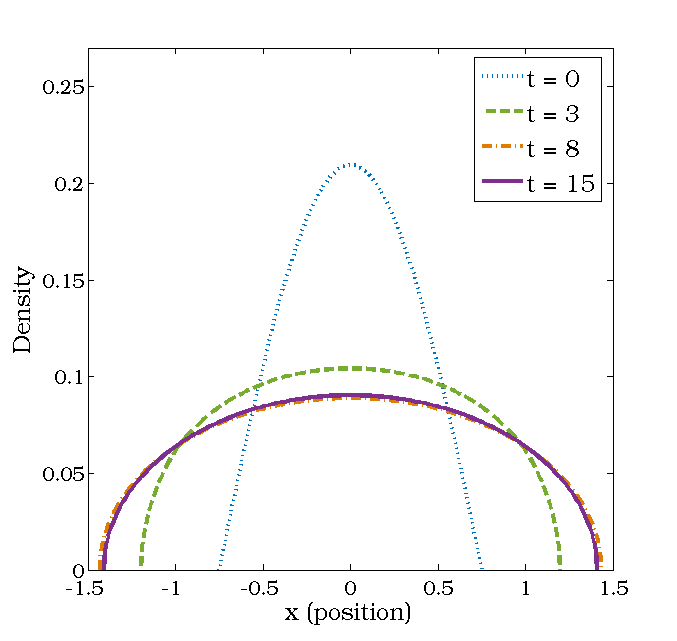}
\llap{\shortstack{%
        \includegraphics[scale=.085]{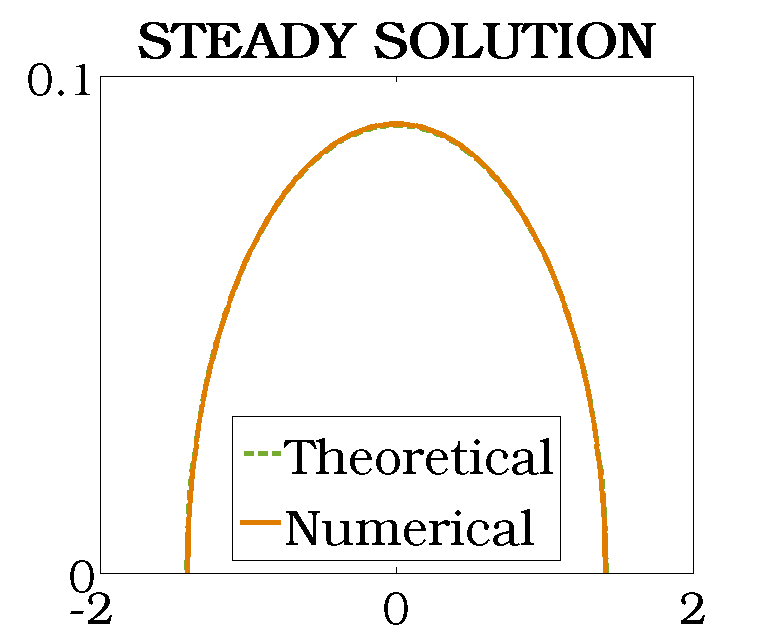}\\
        \rule{0ex}{1.32in}%
      }
  \rule{1.27in}{0ex}}
}\subfloat[]{\protect\includegraphics[scale=0.32]{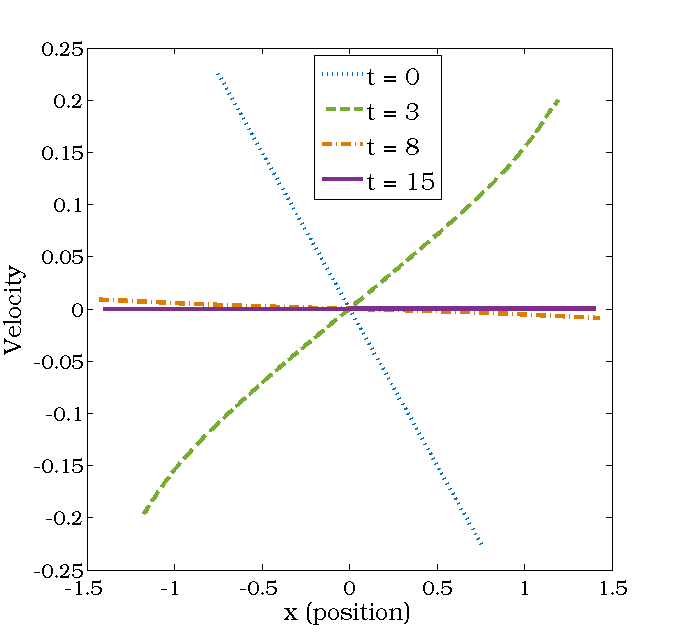}
}\newline
\subfloat[]{
\protect\includegraphics[scale=0.32]{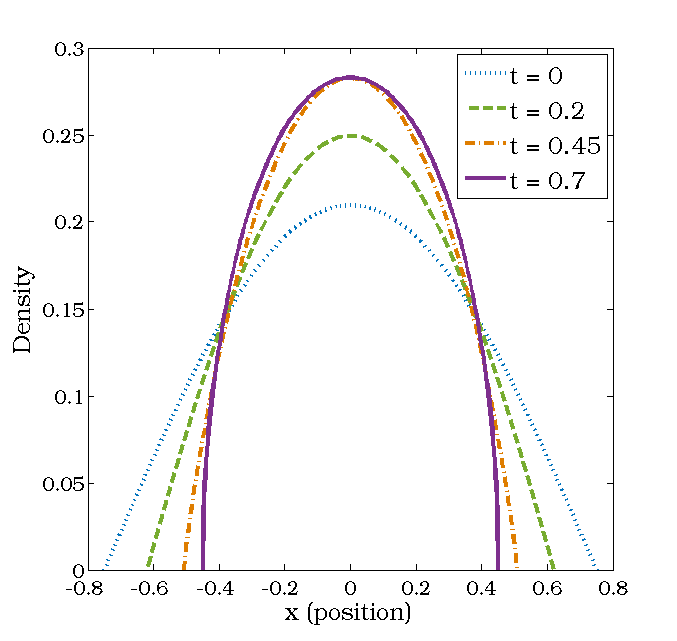}
}\subfloat[]{\protect\includegraphics[scale=0.32]{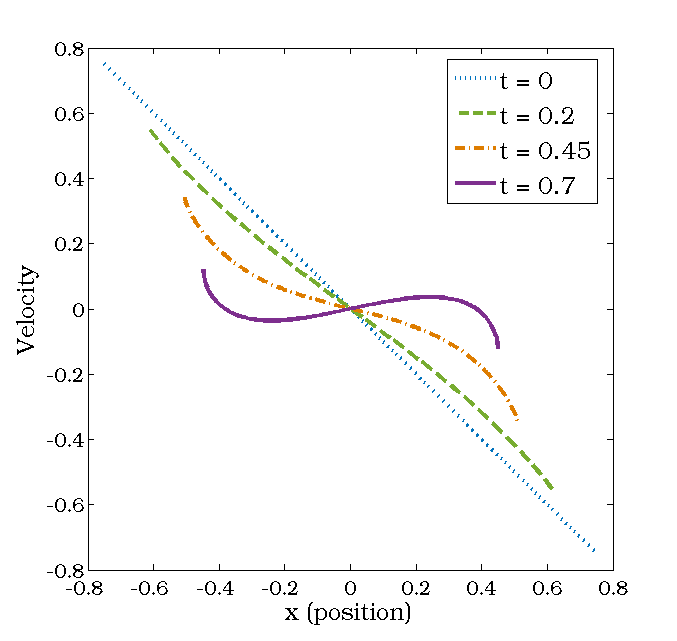}
\llap{\shortstack{%
        \includegraphics[scale=.1]{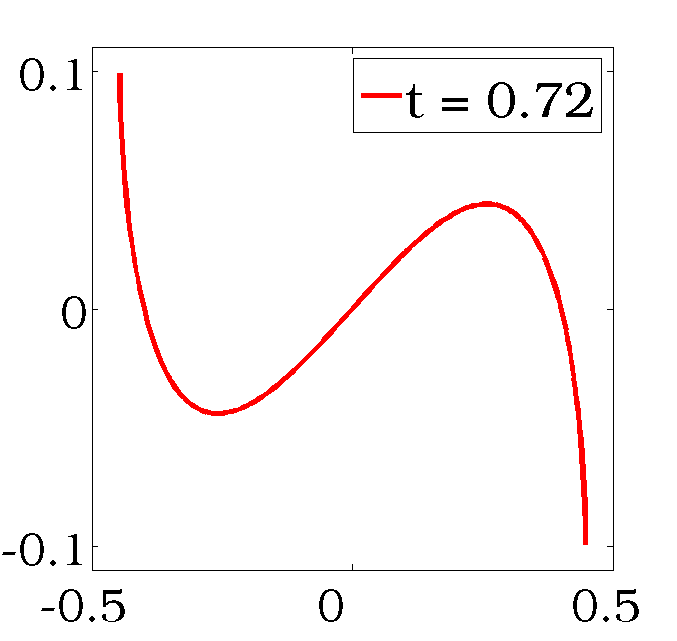}\\
        \rule{0ex}{0.23in}%
      }
  \rule{1.23in}{0ex}}
}\protect\caption{\label{fig:nonnew_blowup} Numerical simulations of the system \eqref{h_np} with potential \eqref{pot_nnp}.- (a), (b): Time behavior of the density and the velocity for the case $M_0=0.2$ and $c=0.3$. (c), (d): Time behavior of the density and the velocity for the case $M_0=0.2$ and $c=1$.}
\end{figure}

In Fig. \ref{fig:nonnew_blowup}, two numerical simulations are depicted with the objective of supporting Theorem \ref{thm_blowupnonnew}. The first case, with $M_0=0.2$ and $c=0.3$, the initial conditions of the blow up estimate of Theorem \ref{thm_blowupnonnew} are not satisfied. The numerical experiment seems to indicate that there exists a global-in-time solution converging toward the steady solution given by the semicircle law. We do not have sharp critical thresholds for this system so its behavior could not be predicted beforehand. On the other hand, in the second case it is set $M_0=0.2$ and $c=1$, and those initial conditions satisfy the blow up estimate of Theorem \ref{thm_blowupnonnew}. It can be observed in Fig. \ref{fig:nonnew_blowup} (f) that the finite-time blow-up is again produced by the infinite slope of the velocity at the boundary. Finding critical thresholds for this case is an open problem.

%

To conclude, a numerical study of the CS system \eqref{h_CS_P} with \eqref{pot_nnp} has been carried out. In fact, as above the total mass of the system and the initial conditions affect the global behavior of the solution leading to global existence of solutions converging to the semicircle law or finite time blow-up. The qualitative behavior is similar to the case of linear damping in case of initial data with zero mean velocity. Otherwise, the system can lead to travelling wave solutions with the same density profile.

\subsubsection{Euler-aligment system with attractive power-law potential and pressure}
In this part, we consider the pressure term $p$ in the Euler-alignment system with the attractive power-law potential $K(x) = x^2/2$ leading to the system
\begin{equation}\label{press}
\begin{cases}
\partial_{t}\rho+\pa_x \left(\rho u\right)=0,\quad x\in\R,\quad t>0,\\[2mm]
{\displaystyle \pa_{t}(\rho u)+ \pa_x(\rho u^2)+\pa_x p(\rho)= \rho\lt[\psi*(\rho u)-(\psi*\rho)u  - \pa_x K \star \rho\rt],}
\end{cases}
\end{equation}
where the pressure-law is given by $p(\rho) = \rho^2$.

For the numerical approximation, we will use a variation of the Lagragian scheme as in Section \ref{subsec:numerics}. Notice that we can not directly apply that scheme due to the presence of pressure $p$. In order to overcome this new difficulty, we take into account the approximated pressure term $\pa_x p_\epsilon(\rho) := 2\rho(\pa_x \delta_\epsilon  \star \rho)$, where $\delta_\epsilon$ is a mollification of the Dirac delta function $\delta_0$ given by a Gaussian
\[
\delta_\epsilon(x) = \frac{1}{\sqrt{2\pi\epsilon}}e^{\frac{-x^2}{2\epsilon}}.
\]
Note that $\delta_\epsilon$ converges weakly as mesures to $\delta_0$ as $\epsilon \to 0$, and subsequently, this implies 
$\pa_x p_\epsilon(\rho) = 2\rho(\delta_\epsilon \star \pa_x \rho) \to 2\rho\pa_x \rho = \pa_x p(\rho)$ as $\epsilon \to 0$, for smooth enough mass densities $\rho$. Using this approximation, we can rewrite the system $\eqref{press}_2$ as
\[
\pa_{t}(\rho u)+ \pa_x(\rho u^2) = \int_{\R}\psi(x-y)(u(y)-u(x))\rho(x)\rho(y)\,dy  - \rho \lt(\pa_x \tilde K \star \rho\rt),
\]
where the interaction potential $\tilde K$ is given by
\[
\tilde K(x) = -\frac{1}{\sqrt{2\pi \epsilon}}e^{\frac{-x^2}{2\epsilon}} + \frac{x^2}{2}.
\]
This enables us to use the previous Lagrangian scheme for the numerical simulations. It is worth mentioning that the parameter $\epsilon$ can not be too small for the numerical simulation and it should be chosen according to both the number of nodes and the distance between them. In our setting, we take the parameter $\epsilon = 10^{-4.1}$ for the numerical simulations.

Similarly as before, we consider the CS nonlocal velocity alignment force and the linear damping for the numerical simulations. We again remind the reader that if we choose the constant communication function $\psi \equiv 1$, then CS velocity alignment force becomes the linear damping by assuming that the initial momentum is zero, i.e., $\int_\R (\rho_0 u_0)(x)\,dx = 0$. 

Let us next investigate the steady solution for the system \eqref{press}. Let us first look for steady solutions of the form $\rho = \rho_\infty(x)$ and $u = u_\infty \equiv 0$. Since the initial momentum is zero, it is straightforward to check that the center of mass of the density $\rho$ is preserved on time. Let us assume without loss of generality that the center of mass is zero. Plugging $\rho_\infty$ and $u_\infty$ into \eqref{press}, we easily find 
\[
2\pa_x \rho_\infty(x) = -(x \star \rho_\infty)(x) = -x M_0\quad \mbox{on} \quad \mbox{supp}(\rho_\infty).
\]
This yields, by solving the ODE and fixing the mass to be $M_0$ with zero center of mass, that
\[
\rho_\infty (x) = \left\{ \begin{array}{ll}
\displaystyle -\frac{M_0}{4}\left(x+\sqrt[3]{3}\right)\left(x-\sqrt[3]{3}\right) \quad & \textrm{for} \quad x\in\left[-\sqrt[3]{3},\sqrt[3]{3}\right],\\[3mm]
 0 & \textrm{otherwise}.
  \end{array} \right.
\]

The initial density and velocity for the numerical simulations are defined as
\[
\rho_{i}(0)=\frac{1}{\gamma}\cos\left(\pi\,\frac{x_{i}(0)}{1.5}\right) \quad \mbox{and} \quad u_{i}(0)=-c\text{\,}\sin\left(\pi\,\frac{x_{i}(0)}{1.5}\right),
\]
for each node $i=1,\cdots,n$, where $\gamma$ is chosen so that the mass of the system is unit, and $c=0.2$. Actually, in our numerical experiments we add to the initial data the positive constant $\alpha=0.05$. The presence of vacuum areas lead to numerical artifacts at the boundary if we strictly impose zero value of the density. This is mainly due to the pressure term since the density is expected to become non-differentiable at the tip of the support, as for the steady state $\rho_\infty$. Therefore, we opt by adding this small constant to our initial data over the whole simulation interval. Furthermore, for the CS nonlocal velocity alignment force we have added a constant value to the initial velocity of $0.1$, in order to have a nonzero initial momentum.

\begin{figure}[ht!]
\subfloat[]{
\protect\includegraphics[scale=0.32]{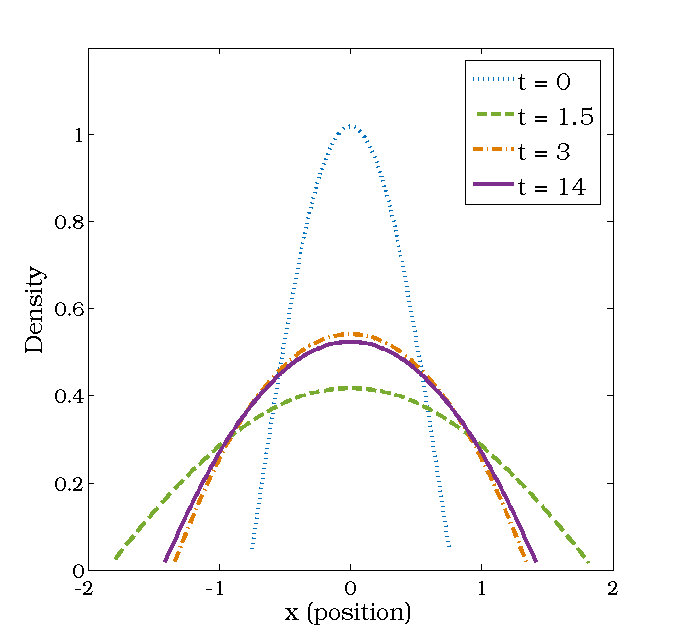}
\llap{\shortstack{%
        \includegraphics[scale=.093]{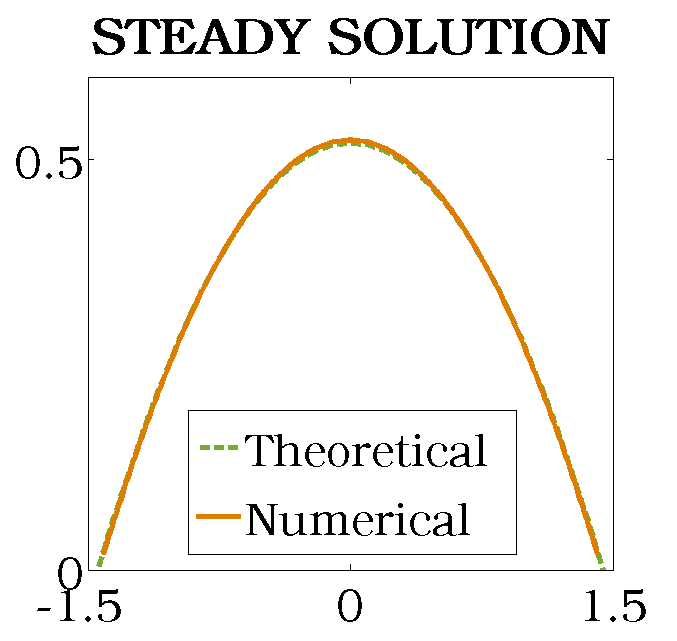}\\
        \rule{0ex}{1.28in}%
      }
  \rule{1.28in}{0ex}}
}\subfloat[]{\protect\includegraphics[scale=0.32]{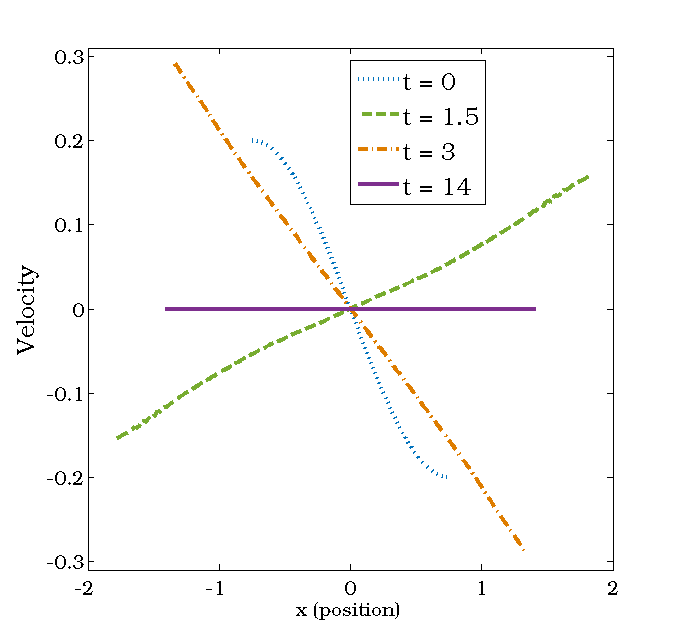}
}\newline
\subfloat[]{
\protect\includegraphics[scale=0.32]{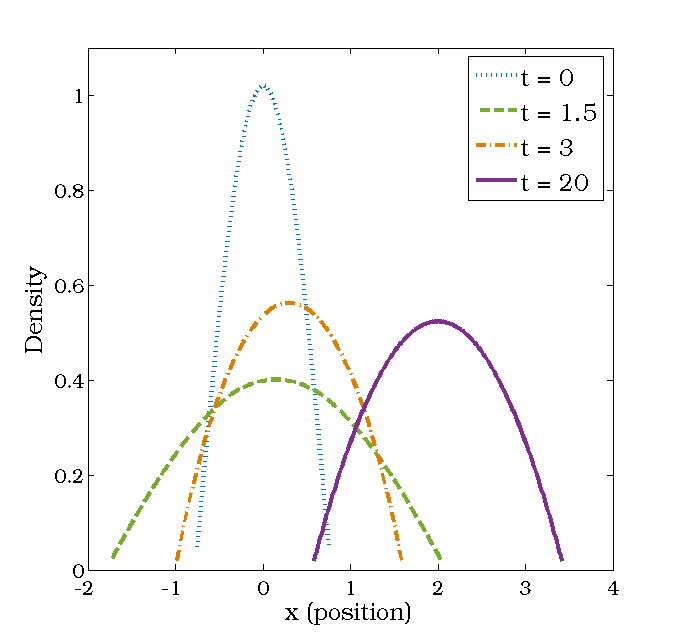}
}\subfloat[]{\protect\includegraphics[scale=0.32]{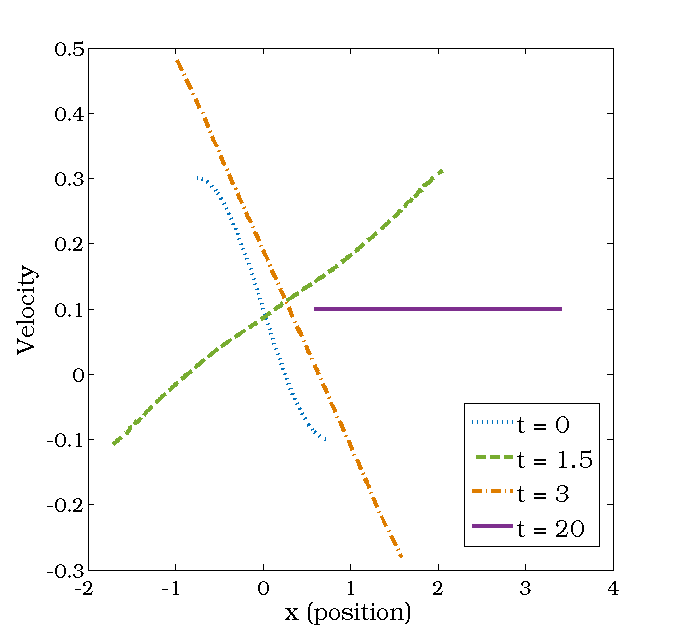}
}\protect\protect\caption{\label{fig:press}Numerical simulation of the density
and velocity with the linear damping (a), (b) and the CS velocity alignment (c), (d).}
\end{figure}

Fig. \ref{fig:press} shows the time evolutions of the density and the velocity for the approximated system with the linear damping (a), (b) and the CS nonlocal velocity alignment force (c), (d). We observe a convergence to the steady state for linear damping, while for the CS model a convergence towards a travelling wave profile due to the nonzero initial momentum. In both cases, we observe very fast convergences toward the steady state/travelling wave with time modulated decaying oscillations. Furthermore, it shows that the shape of the asymptotic density profiles are consistent with the theoretical results.

\section*{Acknowledgements}
J. A. C. was partially supported by the Royal Society via a Wolfson Research Merit Award. J. A. C. and Y.-P. C. were partially supported by EPSRC grant EP/K008404/1. Y.-P. C. was supported by the ERC-Stating grant HDSPCONTR ``High-Dimensional Sparse Optimal Control''. S.P. was partially supported by a Erasmus+ scholarship.

%
%
%

\end{document}